\documentclass[a4paper, reqno, 12pt]{amsart}

\usepackage[usenames,dvipsnames]{color}
\usepackage{amsthm,amsfonts,amssymb,amsmath,amsxtra}

\usepackage{tikz-cd}
\usepackage{todonotes,cancel}

\usepackage[all]{xy}
\SelectTips{cm}{}
\usepackage{xr-hyper}
\usepackage[colorlinks=
   citecolor=Black,
   linkcolor=Red,
   urlcolor=Blue]{hyperref}
\usepackage{verbatim}

\usepackage[margin=1.25in]{geometry}
\usepackage{mathrsfs}

\RequirePackage{xspace}
\RequirePackage{etoolbox}
\RequirePackage{varwidth}
\RequirePackage{enumitem}
\RequirePackage{tensor}
\RequirePackage{mathtools}
\RequirePackage{longtable}
\RequirePackage{multirow}

\def\ge{\geqslant}
\def\le{\leqslant}
\def\a{\alpha}

\def\D{\Delta}

\def\e{\epsilon}

\def\o{\omega}

\def\s{\sigma}

\def\l{\lambda}

\def\i{^{-1}}
\def\ui{\underline i}
\def\uj{\underline j}

\def\<{\langle}
\def\>{\rangle}

\newcommand{\fkG}{\ensuremath{\mathfrak{G}}\xspace}

\newcommand{\fkP}{\ensuremath{\mathfrak{P}}\xspace}

\newcommand{\fkT}{\ensuremath{\mathfrak{T}}\xspace}
\newcommand{\fkU}{\ensuremath{\mathfrak{U}}\xspace}
\newcommand{\fkV}{\ensuremath{\mathfrak{V}}\xspace}

\newcommand{{\BG}}{\ensuremath{\mathbb {G}}\xspace}

\newcommand{{\BK}}{\ensuremath{\mathbb {K}}\xspace}

\newcommand{\BN}{\ensuremath{\mathbb {N}}\xspace}

\newcommand{\BR}{\ensuremath{\mathbb {R}}\xspace}

\newcommand{\BZ}{\ensuremath{\mathbb {Z}}\xspace}

\newcommand{\CB}{\ensuremath{\mathcal {B}}\xspace}

\newcommand{\CD}{\ensuremath{\mathcal {D}}\xspace}

\newcommand{\CR}{\ensuremath{\mathcal {R}}\xspace}

\newcommand{\CT}{\ensuremath{\mathcal {T}}\xspace}

\DeclareMathOperator{\End}{End}

\newcommand{\B}{{\bf B}}
\newcommand{\U}{{\mathbf U}}
\newcommand{\mA}{\mathcal{A}}

\newcommand{\Vla}{{}^\lambda V}
\newcommand{\vw}{\mathbf v}
\newcommand{\vplus}{\mathbf v_{  +}}

\def\kk{\mathbf k}

\newtheorem{theorem}{Theorem}
\newtheorem{prop}[theorem]{Proposition}
\newtheorem{proposition}[theorem]{Proposition}
\newtheorem{lem}[theorem]{Lemma}
\newtheorem{lemma}[theorem]{Lemma}

\newtheorem{cor}[theorem]{Corollary}
\newtheorem{corollary}[theorem]{Corollary}

\theoremstyle{definition}
\newtheorem{definition}[theorem]{Definition}

\newtheorem{rem}[theorem]{Remark}
\newtheorem{remark}[theorem]{Remark}

\numberwithin{equation}{section}
\numberwithin{theorem}{section}

\setitemize[0]{leftmargin=*,itemsep=\the\smallskipamount}
\setenumerate[0]{leftmargin=*,itemsep=\the\smallskipamount}

\renewcommand{\to}{%
   \ifbool{@display}{\longrightarrow}{\rightarrow}%
   }
\let\shortmapsto\mapsto
\renewcommand{\mapsto}{%
   \ifbool{@display}{\longmapsto}{\shortmapsto}%
   }
\newlength{\olen}
\newlength{\ulen}
\newlength{\xlen}
\newcommand{\xra}[2][]{%
   \ifbool{@display}%
      {\settowidth{\olen}{$\overset{#2}{\longrightarrow}$}%
       \settowidth{\ulen}{$\underset{#1}{\longrightarrow}$}%
       \settowidth{\xlen}{$\xrightarrow[#1]{#2}$}%
       \ifdimgreater{\olen}{\xlen}%
          {\underset{#1}{\overset{#2}{\longrightarrow}}}%
          {\ifdimgreater{\ulen}{\xlen}%
             {\underset{#1}{\overset{#2}{\longrightarrow}}}
             {\xrightarrow[#1]{#2}}}}%
      {\xrightarrow[#1]{#2}}
   }
\makeatother
\newcommand{\xyra}[2][]{%
   \settowidth{\xlen}{$\xrightarrow[#1]{#2}$}%
   \ifbool{@display}%
      {\settowidth{\olen}{$\overset{#2}{\longrightarrow}$}%
       \settowidth{\ulen}{$\underset{#1}{\longrightarrow}$}%
       \ifdimgreater{\olen}{\xlen}%
          {\mathrel{\xymatrix@M=.12ex@C=3.2ex{\ar[r]^-{#2}_-{#1} &}}}%
          {\ifdimgreater{\ulen}{\xlen}%
             {\mathrel{\xymatrix@M=.12ex@C=3.2ex{\ar[r]^-{#2}_-{#1} &}}}
             {\mathrel{\xymatrix@M=.12ex@C=\the\xlen{\ar[r]^-{#2}_-{#1} &}}}}}%
      {\mathrel{\xymatrix@M=.12ex@C=\the\xlen{\ar[r]^-{#2}_-{#1} &}}}%
   }
\makeatletter
\newcommand{\xla}[2][]{%
   \ifbool{@display}%
      {\settowidth{\olen}{$\overset{#2}{\longleftarrow}$}%
       \settowidth{\ulen}{$\underset{#1}{\longleftarrow}$}%
       \settowidth{\xlen}{$\xleftarrow[#1]{#2}$}%
       \ifdimgreater{\olen}{\xlen}%
          {\underset{#1}{\overset{#2}{\longleftarrow}}}%
          {\ifdimgreater{\ulen}{\xlen}%
             {\underset{#1}{\overset{#2}{\longleftarrow}}}
             {\xleftarrow[#1]{#2}}}}%
      {\xleftarrow[#1]{#2}}
   }
\newcommand{\isoarrow}{%
   \ifbool{@display}{\overset{\sim}{\longrightarrow}}{\xrightarrow\sim}%
   }
   
\begin{document}

\title[]{Flag manifolds over  semifields}
\author[Huanchen Bao]{Huanchen Bao}
\address{Department of Mathematics, National University of Singapore, Singapore.}
\email{huanchen@nus.edu.sg}

\author[Xuhua He]{Xuhua He}
\address{The Institute of Mathematical Sciences and Department of Mathematics, The Chinese University of Hong Kong, Shatin, N.T., Hong Kong.}
\email{xuhuahe@gmail.com}
\thanks{}


\date{\today}

\begin{abstract}
In this paper, we develop the theory of flag manifold over a semifield for any Kac-Moody root datum. We show that the flag manifold over a semifield admits a natural action of the monoid over that semifield associated with the  Kac-Moody datum and admits a cellular decomposition. This extends the previous work of Lusztig, Postnikov, Rietsch and others on the totally nonnegative flag manifolds (of finite type) and the work of Lusztig, Speyer, Williams on the tropical flag manifolds (of finite type). As a by-product, we prove a conjecture of Lusztig on the duality of totally nonnegative flag manifold of finite type. 
\end{abstract}

\maketitle

\tableofcontents

\section{Introduction}

\subsection{The theory of total positivity}
By definition, a matrix in $GL_n(\BR)$ is called totally positive (resp. totally nonnegative) if all its minors are positive (resp. nonnegative). The theory of totally positive real matrices was originated in the 1930's by Schoenberg \cite{Sch}, and by Gantmacher and Krein \cite{GK} after earlier contribution by Fekete and Polya in 1912. It was further developed by Whitney and Loewner in the 1950's. 

The group of invertible matrices is a special case of the split reductive groups. In the foundational work \cite{Lus-1}, Lusztig developed the theory of total positivity for arbitrary split real reductive group $G$. The totally nonnegative part $G(\BR_{>0}) =G_{\ge 0}$ of $G(\BR)$ forms a monoid under the multiplication in $G$. Lusztig showed that $G(\BR_{>0})$ admits a cellular decomposition indexed by the pairs of elements in the Weyl group $W$ of $G$. 

Lusztig then defined the totally nonnegative flag manifold $\mathcal B(\BR_{>0}) = \mathcal B_{\ge 0}$. This is a certain subset of the  flag manifold $\mathcal B(\BR)$ which is stable under the natural monoid action of $G(\BR_{>0})$ on $\mathcal B(\BR)$. Lusztig in \cite{Lus-1} conjectured that the totally nonnegative flag manifolds admit cellular decomposition and the cells are indexed by the pairs $v \le w$ in the Weyl group $W$. This was proved by Rietsch in \cite{Ri99} and an explicit parametrization of each cell was obtained by Marsh and Rietsch in \cite{MR}. 
These approaches to totally nonnegative flag manifolds uses crucially the topology on $\BR$ (so that one may take the limit of a sequence, etc.). The construction can be generalized to partial flag manifolds.

Lusztig's theory of total positivity has important applications in different areas, including the theory of cluster algebras introduced by Fomin and Zelevinsky \cite{FZ}; higher Teichm\"uller theory by Fock and Goncharov \cite{FG}; the theory of amplituhedron in physics by Arkani-Hamed and Trnka \cite{AHT}, etc.  The combinatorial aspects of the totally nonnegative Grassmannian was also studied extensively by Postnikov in \cite{Pos}. 

\subsection{Flag manifolds over semifields}
In fact, $\BR_{>0}$ is an example of semifields (a terminology of Berenstein, Fomin and Zelevinsky \cite{BFZ}). Other important examples of semifields include the tropical semifield $(\BZ^{trop}, \min, +)$, which plays a crucial role in the tropical geometry; and the semifield $\{1\}$ of one element.   The tropicalization of totally nonnegative Grassmannian was already studied by Speyer and Williams in \cite{SW}. It is desirable to generalize the theory of total positivity from $\BR_{>0}$ to any semifield. It is also desirable to generalize the theory from root data of finite type to arbitrary (symmetrizable) Kac-Moody root data (even over $\BR_{>0}$). 
 
Kac-Moody groups come in two different versions (minimal and maximal). In the finite type, the minimal and maximal Kac-Moody groups coincide. But in general, they are very different. In \cite{Lu-positive}, Lusztig constructed the (minimal) monoid $\mathfrak G(K)$ over any semifield $K$ associated with any Kac-Moody root datum. As to the maximal Kac-Moody groups, Lam and Pylyavskyy \cite{LP} developed a theory of total positivity for loop groups (the maximal affine Kac-Moody groups) and the object there is very different from what we consider here. The flag manifolds associated to the minimal and maximal Kac-Moody groups, are naturally bijective as sets.

 

The study of the flag manifold $\CB(K)$ over arbitrary semifield $K$ and associated with any Kac-Moody data was initiated by Lusztig in the sequel papers \cite{Lu-positive}, \cite{Lu-2}, \cite{Lu-Spr}, \cite{Lu-flag} and \cite{Lu-par}. The two important features one would like to have are 


\begin{itemize}
	\item 	The canonical decomposition of $\CB(K)$ into cells ($\cong K^{n}$);
	\item 	 The natural monoid action of $\mathfrak G(K)$ on $\CB(K)$.
\end{itemize}

In \cite{Lu-2, Lu-Spr}, Lusztig gave a definition of the flag manifold $\CB(K)$ of finite type as the disjoint union of the cells over $K$ based on the map introduced by Marsh-Rietsch in \cite{MR}. However, in the definition, the lower and upper triangular parts of the monoid $\mathfrak G(K)$ play  asymmetric roles. It was not known in loc.cit. that the whole monoid $\mathfrak G(K)$ acts naturally on $\CB(K)$.


Using the theory of canonical bases, Lusztig gave several other definitions of the flag manifold $\CB(K)$ over a semifield $K$. The flag manifold $\CB(K)$ was first defined over $K=\BR_{>0}$ in \cite{Lu-positive}, and then over $\BR(t)_{>0}$ and over $\BZ^{trop}$ in \cite{Lu-flag}. The construction involves  a (single) highest weight module. In \cite{Lu-par}, Lusztig gave a different definition, which works for any semifield, based on the tensor product structure among all  irreducible highest weight modules. All these definitions work for Kac-Moody root data. The $\fkG(K)$-monoid action follows naturally from the theory of canonical bases. The cellular decomposition is known for $\CB(\BR_{>0})$ for the finite type case. Using the topology of $\BR$, Lusztig proved that for the finite type cases, the flag manifolds $\CB(\BR(t)_{>0})$ and $\CB(\BZ^{trop})$ defined in \cite{Lu-flag} coincide with the ones defined in \cite{Lu-Spr} and hence also have the cellular decomposition. 

In the remaining cases (for finite type cases over other semifields and for non-finite type Kac-Moody cases), the cellular decomposition remains highly nontrivial. 

\subsection{Main results} 
The main result of this paper is the following theorem:

\begin{theorem}[See Theorem~\ref{thm:flagK}] \label{1.1}
Let $K$ be a semifield.  
\begin{enumerate}
\item (Definition) We define the flag manifold $\CB(K)$ over an arbitrary semifield $K$ for any (symmetrizable) Kac-Moody root datum. 
\item (Cellular decomposition) The flag manifold $\CB(K)$ admits a canonical partition $\CB(K)=\sqcup_{v \le w \text{ in } W} \CB_{v, w}(K)$ and each piece $\CB_{v, w}(K)$ is in bijection with $K^{\ell(w)-\ell(v)}$. Here $\ell(-)$ is the length function on $W$. 

\item (Monoid action) The flag manifold $\CB(K)$ has a natural action of the monoid $\mathfrak G(K)$.

\item (Base change) The monoid action and the cell decomposition on the flag manifold $\CB(-)$ are compatible with the base change induced by the homomorphism of semifields. 
\end{enumerate}
\end{theorem} 

 Our definition of $\CB(K)$ coincides with Lusztig's definition in the case when the semifield $K$ is $\mathbb{R}_{>0}$ (\cite{Lu-positive}), $\mathbb{R}(t)_{>0}$ (\cite{Lu-flag}), and the tropical $\mathbb{Z}^{trop}$(\cite{Lu-flag}). It also coincides with Lusztig's definition in \cite{Lu-par} when the semifield $K$ is contained in a field. It is an interesting question whether our definition of $\CB(K)$ coincides with Lusztig's definition in \cite{Lu-par} for arbitrary semifields.  


\subsection{Applications}
We first list some interesting special cases.

\begin{itemize}
\item The positive real number $K=\BR_{>0}$ case: we have the totally nonnegative flag manifolds for any Kac-Moody groups, which admits the cellular decomposition and admits the action of the totally nonnegative part $  G(\BR_{>0})$ of the Kac-Moody group. The totally nonnegative affine flag manifold would be of independent interest. 

We also prove in Theorem~\ref{thm:real} that the totally nonnegative flag manifold $\CB(\BR_{>0})$ equals to the closure of $\fkG(\BR_{>0}) B^+(\BR)/B^+(\BR)$ in $\CB(\BR)$ with respect to the usual topology. This generalizes the classical definition of totally nonnegative flag manifolds. Moreover, each cell $\CB_{v,w}(\BR_{>0})$ is a topological cell, i.e.,  is homeomorphic to $\mathbb{R}_{>0}^n$ for some $n$.

\item The tropical semifield $K=\BZ^{trop}$ case: we have the tropical flag manifold $\CB(\BZ^{trop})$. The tropical flag manifold $\CB(\BZ^{trop})$ admits a natural action of the tropical monoid $G(\BZ^{trop})$ and admits a cellular decomposition.  

\item The semifield of one element $K=\{1\}$ case: the totally nonnegative flag manifold $\CB(\{1\})$ gives the index set of the cell decomposition of $\CB(K')$ for any semifield $K'$ and this index set admits a natural action of the monoid $W^\sharp \times W^\sharp$ over $\{1\}$. Here $W^\sharp$ is the monoid associated to the Weyl group $W$ of the Kac-Moody group $G$. 
\end{itemize}


As an application of the base change (Theorem \ref{1.1} (3)), we show that 

\begin{proposition}[See \S\ref{sec:flag1}]
The natural monoid action $\fkG(K)$ on $\CB(K)$ is compatible with the cellular decomposition, i.e., for any cell $C_1$ of $\fkG(K)$ and $C_2$ of $\CB(K)$ in the cellular decomposition in Theorem \ref{1.1}, the monoid action sends $(C_1, C_2)$ to a cell in $\CB(K)$. 
\end{proposition}

Note that the index set of the cellular decompositions of $\fkG(K)$ and $\CB(K)$ are given by $\fkG(\{1\})$ and $\CB(\{1\})$ respectively. We also give an explicit description of this monoid action on the index sets of the cellular decompositions of $\fkG(K)$ and $\CB(K)$.

\subsection{The strategy} 

We first study the case where the Kac-Moody root datum is symmetric and the semifield $K$ is contained in a field $\kk$. In this case, we follow Lusztig's definition in \cite{Lu-flag} based on the theory of canonical bases. For any dominant regular weight $\l$, let ${}^{\l} V(\kk)$ be the irreducible highest weight representation of the minimal Kac-Moody group $G^{\min}(\kk)$ and ${}^{\l} P(\kk)$ be the projective space of ${}^{\l} V(\kk)$. Let ${}^{\l} P(K)$ be the subset of ${}^{\l} P(\kk)$ consisting of lines spanned by vectors in ${}^{\l} V(\kk)$ whose coefficients with respect to the canonical basis are all in $K \cup \{0\}$. Let ${}^{\l} P^{\bullet} (K)$ be the intersection of ${}^{\l} P(K)$ with the image of $\CB(\kk)$ in ${}^{\l} P(\kk)$. Then ${}^{\l} P^{\bullet} (K)$ admits a natural action of the  monoid $\mathfrak G(K)$. The nontrivial part is to show that it admits a decomposition into the Marsh-Rietsch cells and is independent of the choice of $\l$. This is obtained from a detailed study of the relation between the canonical basis of Lusztig and the chamber ansatz of Berenstein-Zelevinsky and Marsh-Rietsch. The case where the semifield $K$ is contained in a field $\kk$ for arbitrary Kac-Moody root datum then follows from the symmetric case via the ``folding method'' of Lusztig. This is done in section \ref{sec:3}. 

There are additional difficulties when considering general semifields. Note that there is no $\CB(\kk)$ or  $G^{\min}(\kk)$ in the general case and thus the definition of ${}^{\l} P^{\bullet}(K)$ above does not work. We define ${}^{\l} P^{\bullet}(K)$ via base change from the case when the semifield is contained in a field. We also obtain the canonical partition ${}^{\l} P^{\bullet}(K) = \sqcup_{v \le w}{}^{\l} P_{v,w}^{\bullet}(K)$ using canonical bases.
A priori, this definition depends on the choice of $\lambda$. 

We then construct explicit bijections from $K^{\ell(w) - \ell(v)}$ to ${}^{\l} P_{v,w}^{\bullet}(K)$ motivated by Marsh-Rietsch's construction for totally nonnegative flag manifold $\CB(\BR_{>0})$. Due the lack of the ambient Kac-Moody group $ G^{\min}(\kk)$, we replace the group-theoretical operator $\dot{s}_i$ by a set-theoretical bijection $\tilde{s}_{i}$ of canonical bases. 
In this way, we also see how the different cells in ${}^{\l} P^{\bullet}(K)$ are putting together, and hence also remove the dependence on the dominant regular $\lambda$. In this way, we obtain the flag manifold $\CB(K) = \sqcup_{v,w}\CB_{v,w}(K)$ with desired properties.





In section \ref{sec:5}, we study the ``coordinate charts'' on each piece $\CB_{v, w}(K)$ and show that the transition maps among different charts are admissible in the sense of Lusztig \cite{Lu-2}. Finally, in Proposition \ref{prop:conj}, we  verify Lusztig's conjecture on the admissibility of the duality $\phi$ on the flag manifolds of finite type. 

\vspace{.4cm}
 {\bf Acknowledgement: }We would like to thank Thomas Lam, Lauren Williams and George Lusztig for helpful comments and suggestions. HB is supported by a NUS start-up grant. XH is partially supported by a start-up grant and by funds connected with Choh-Ming Chair at CUHK.
 



\section{Groups and monoids of Kac-Moody type}

\subsection{Kac-Moody root datum}

Let $I$ be a finite set and $A=(a_{ij})_{i, j \in I}$ be a symmetrizable generalized Cartan matrix in the sense of \cite[\S 1.1]{Kac}. A {\it Kac-Moody root datum} associated to $A$ is a quintuple $$\CD=(I, A, X, Y, (\a_i)_{i \in I}, (\a^\vee_i)_{i \in I}),$$ where $X$ is a free $\BZ$-module of finite rank with $\BZ$-dual $Y$, and the elements $\a_i$ of $X$ and $\a^\vee_i$ of $Y$ such that $\<\a^\vee_j, \a_i\>=a_{ij}$ for $i, j \in I$. We denote by $\omega_i \in X$  the element that $\<\a^\vee_j, \omega_i\>=\delta_{ij}$.

Let $m_{i j} (i,j \in I)$ be positive integers or $\infty$ defined by the following table: 
\[
\begin{tabular}{c|c c c c c}
$\langle \alpha^\vee_i, \alpha_j \rangle \langle \alpha^\vee_j, \alpha_i \rangle $ & $0$ & $1$ & $2$ & $3$ & $\ge 4$\\
\hline
$m_{i j}$ & $2$ & $3$ & $4$ & $6$ & $\infty$. 
\end{tabular}
\]

Let $W$ be the corresponding Weyl group. It is the group generated by the simple reflections $s_i$ for $i \in I$, subject to the relations
\begin{itemize}
\item For any $i \in I$, $s_i^2=1$; 

\item For any $i \neq j \in I$ with $m_{i j}$ finite, $s_i s_j \cdots=s_j s_i \cdots$ (both products have $m_{i j}$ factors). 
\end{itemize}

We have natural actions of $W$ on both $X$ and $Y$. Let $$\Delta^{re} = \{w ( \pm \alpha_i) \in X \vert i \in I, w \in W\} \subset X$$ be the set of real roots. Then $\D^{re}=\D^{re}_+ \sqcup \D^{re}_-$ is the union of positive real roots and negative real roots. 

We say the root datum is {\it simply connected} if $Y = \mathbb{Z}[\alpha^\vee_i]_{i \in I}$. We say the root datum is {\it symmetric} if $A$ is symmetric. 
Note that when the root datum is symmetric, we have $m_{i,j} \in \{2,3, \infty\}$. We say the root datum is {\it of finite type} if $W$ is a finite group.

\subsection{Minimal Kac-Moody groups} Let $\kk$ be a field. The {\it minimal Kac-Moody group} $G^{\min}(\kk)$ associated to the Kac-Moody root datum $\CD$ is the group generated by the torus $T(\kk)=Y \otimes_\BZ \kk^\times$ and the root subgroup $U_\a(\kk) \cong \kk$ for each real root $\a$, subject to the Tits relations \cite{Ti87}. Let $U^+(\kk) \subset G^{\min}(\kk)$ (resp. $U^-(\kk) \subset G^{\min}(\kk)$) be the subgroup generated by $U_\a(\kk)$ for $\a \in \D^{re}_+$ (resp. $\a \in \D^{re}_-$). Let $B^{\pm}(\kk) \subset G^{\min}(\kk)$ be the Borel subgroup generated by $T(\kk)$ and $U^{\pm}(\kk)$. 

We fix the pinning $(T(\kk), B^+(\kk), B^-(\kk), x_i, y_i; i\in I)$ of $G^{\min}(\kk)$. It consists of isomorphisms $x_i: \kk \to U_{\a_i}(\kk)$ and $y_i: \kk \to U_{-\a_i}(\kk)$ for each $i \in I$ such that the maps 
\[
\begin{pmatrix} 1 & a  \\ 0 & 1 \end{pmatrix} \mapsto x_i(a), \begin{pmatrix} b & 0  \\ 0 & b \i \end{pmatrix} \mapsto \a^\vee_i(a), \begin{pmatrix} 1 & 0  \\ c & 1 \end{pmatrix} \mapsto y_i(c)
\] defines a homomorphism $SL_2(\kk) \to G^{\min}(\kk)$. 

For any $i \in I$, we set $\dot s_i=x_i(1) y_i(-1) x_i(1) \in G^{\min}(\kk)$. Let $w \in W$. By \cite[Proposition 7.57]{Mar}, for any reduced expression $w=s_{i_1} s_{i_2} \cdots$ of $w$, the element $\dot s_{i_1} \dot s_{i_2} \cdots$ of $G^{\min}(\kk)$ is independent of the choice of the reduced expression. We denote this element by $\dot w$. 




\subsection{Semifields}\label{sec:semifield} By definition, a {\it semifield} $K$ is a set with two operations $+, \times$, which is an abelian group with respect to $\times$, an abelian semigroup with respect to $+$ and in which $(a+b)c=ac+bc$ for all $a, b, c$. We denote by $1$ the multiplicative identity in $K$. The following three examples are considered by Lusztig in \cite{Lu-2}. 

(1) There exists a field $\kk$  such that $K \subset \kk$ (so $\kk$ is necessarily of characteristic $0$ and the additive identity $0 \in \kk$ is not contained in $K$); 

(2) $K=\BZ$ with a new sum $(a, b) \mapsto \min(a, b)$ and a new product $(a, b) \mapsto a+b$. This is the tropical semifield. 

(3) $K=\{1\}$ with $1+1=1$ and $1 \times 1=1$. This is the semifield of one element. 


%

Let $\mathfrak I_K$ be the set of all pairs $(K', r)$, where $K'$ is a semifield that is contained in a field and $r: K' \to K$ is a homomorphism of semifields. As a consequence of  \cite[Lemma 2.1.6]{BFZ}, we have that $$K=\cup_{(K', r) \in \mathfrak I_K} r(K').$$

\subsection{The monoid  $\mathfrak U(K)$}\label{uk}
We recall the definitions in \cite[\S 2.5 \& \S 3.1]{Lu-positive} and \cite[\S 2.9 \& \S 2.10]{Lu-2}. 

\begin{definition}
Let $\mathfrak U(K)$ be the monoid with generators the symbols $i^a$ with $i \in I$ and $a \in K$ and with relations

(i) For $i \in I$ and $a, b \in K$,  $i^a i^b=i^{a+b}$;

(ii) For any $i \neq j \in I$ with $m_{i j}$ finite and $a_1, \ldots, a_{m_{i j}} \in K$, $$i^{a_1} j^{a_2}\cdots=j^{a'_1} i^{a'_2} \cdots \text{ (both products have $m_{i j}$ factors)}. $$ Here $(a'_1, a'_2, \ldots)=R(i, j)(a_1, a_2, \ldots)$, where the map $R(i, j): K^{m_{i j}} \to K^{m_{i j}}$ is the bi-admissible map defined in \cite[\S 2.4]{Lu-2}.
\end{definition}

\begin{definition}

Following  \cite[\S 2.10]{Lu-2}, we define the {\it monoid} $\mathfrak G(K)$ to be the monoid with generators the symbols
$i^a, (-i)^a,\ui^a$ with $i\in I$, $a\in K$ and with relations (i)-(vii) below. 

(i) For $i\in I$, $\e=\pm1$, $a,b \in K$, $(\e i)^a(\e i)^b=(\e i)^{a+b}$ ;

(ii) For any $\e=\pm 1$, $i \neq j \in I$ with $m_{i j}$ finite and $a_1, \ldots, a_{m_{i j}} \in K$, $$(\e i)^{a_1} (\e j)^{a_2}\cdots=(\e j)^{a'_1} (\e i)^{a'_2} \cdots \text{ (both products have $m_{i j}$ factors)},$$ where $(a'_1, a'_2, \ldots)=R(i, j)(a_1, a_2, \ldots)$;

(iii) For $i\in I$, $a,b,c \in K$, $i^a\ui^b(-i)^c=(-i)^{c/(ac+b^2)}\ui^{(ac+b^2)/b}i^{a/(ac+b^2)}$;

(iv) For $i\in I$, $a,b \in K$, $\ui^a\ui^b=\ui^{ab}$, $\ui^{(1)}=1$;

(v) For $i,j$ in $I$, $a,b \in K$, $\ui^a\uj^b=\uj^b\ui^a$;

(vi) For $i,j$ in $I$, $\e=\pm1$, $a,b$ in $K$, $\uj^a(\e i)^b=(\e i)^{a^{\e\<j,i^*\>}b}\uj^a$;

(vii) For $i\ne j$ in $I$, $\e=\pm1$, $a,b$ in $K$, $(\e i)^a(-\e j)^b=(-\e j)^b(\e i)^a$.

\end{definition}

%
 %
%
%
%


The following result follows easily from the definition. 

\begin{lem}\label{lem:sym}
\begin{enumerate}
	\item There exists a unique monoid automorphism $$\phi: \begin{tikzcd}[column sep =small] \fkG(K) \ar[r,"\simeq"] &\fkG(K),  \end{tikzcd}$$ such that $\phi(i^a) = (-i)^a$, $\phi((-i)^a) = i^a$ and $\phi(\ui^a) = \ui^{a^{-1}}$ for $i \in I$ and $a \in K$. 
	\item There exists a unique monoid anti-automorphism $$\tau: \begin{tikzcd}[column sep =small] \fkG(K) \ar[r,"\simeq"] &\fkG(K),\end{tikzcd}$$ such that $\tau(i^a) = (-i)^a$, $\tau((-i)^a) = i^a$ and $\tau(\ui^a) = \ui^{a}$ for $i \in I$ and $a \in K$. 
\end{enumerate}
\end{lem}
Let $\mathfrak{T}(K)$ be the submonoid generated by $\ui^a$ with $i\in I$, $a\in K$. We may naturally identify the submonoid of $\fkG(K)$ generated by $i^a$ for various $i\in I$ and $a \in K$ with $\fkU(K)$.  Then $\phi$ induces a monoid isomorphism from $\fkU(K)$ to the submonoid of $\fkG(K)$ generated by $(-i)^a$ for various $i \in I$ and $a \in K$.

In the case where $K \subset \kk$, let $U^+(K)$ (resp. $U^-(K)$) be the submonoid of $G^{\min}(\kk)$ generated by $x_i(K)$ (resp. $y_i(K)$) for $i \in I$, $G^{\min}(K)$ be the submonoid of $G^{\min}(\kk)$ generated by $x_i(K)$, $y_i(K)$ for $i \in I$ and $T (K)=Y \otimes_\BZ K \subset T (\kk)$. 

\begin{prop} \label{prop:GG}\cite[\S 2.9 (d) \& \S 2.10 (f)]{Lu-2} Suppose that the semifield $K$ is contained in a field $\kk$. 
Then

(1) The map $i^a \mapsto x_i(a)$ defines an isomorphism of monoids $\mathfrak U(K) \to U^+(K)$. 

(2) The map $i^a \mapsto y_i(a)$ defines an isomorphism of monoids $\mathfrak U(K) \to U^-(K)$. 

(3) The map $i^a \mapsto x_i(a)$, $(-i)^a \mapsto y_i(a)$, $\ui^a \mapsto a^{h_i}$ defines an isomorphism of monoids $\mathfrak G(K) \cong G^{\min}(K)$. 
\end{prop}



\subsection{Canonical bases}

Let ${\bf U}$ be the quantum group associated to the root datum $\mathcal{D}$ generated by $E_i, F_i, K_\mu$ for $i \in I$ and $\mu \in Y$ over the field $\mathbb{Q}(v)$ for an indeterminate $v$. We denote by $E^{(n)}_{i}$ and $F^{(n)}_i$ the divided powers defined in \cite[\S 3.1]{Lu94}. Let $\U^+$ (resp.  $\U^-$ ) be the  $\mathbb{Q}(v)$-subalgebra of $\U$ generated by $E_i$ (resp.  $F_i$) for various $i \in I$.

Let $X^+=\{\l \in X; \<\a_i^\vee, \l\> \ge 0 \text{ for all } i \in I\}$ be the set of dominant weights and $X^{++}=\{\l \in X; \<\a_i^\vee, \l\>>0 \text{ for all } i \in I\}$ be the set of dominant regular weights. 

For any $\lambda \in X^+$, we denote by ${}^\lambda V$ be the integrable highest $\U$-module defined in \cite[Proposition~3.5.6]{Lu94}. Let $\B(\lambda)$ be the canonical basis of $\Vla$.  Let $\eta_\lambda \in {}^\lambda V$ be the highest weight vector. We always assume $\eta_{\lambda} \in \bf{B}(\lambda)$. For any $w \in W$, let $\eta_{w\lambda} \in \bf{B}(\lambda)$ be the extremal vector of weight $w\lambda$. 

Let $v \le w$. We define 
\begin{gather*}
	{}^\lambda V_w = {\bf U}^+ \eta_{w \lambda} \subset   {}^\lambda  V, \quad \text{ and } \quad {\bf B}_w(\lambda) = {\bf B}(\lambda) \cap {}^\lambda  V_w; \\
	\quad {}^\lambda V_{v,w} = {\bf U}^+ \eta_{w \lambda} \cap{\bf U}^- \eta_{v \lambda} \subset   {}^\lambda  V, \quad \text{ and } \quad {\bf B}_{v,w}(\lambda) = {\bf B}(\lambda) \cap {}^\lambda  V_{v,w}.
\end{gather*}
The subspaces ${}^\lambda  V_w$ has basis ${\bf B}_w(\lambda)$ and ${}^\lambda V_{v,w}$ has basis ${\bf B}_{v,w}(\lambda)$ following \cite[Lemma~8.2.1]{Ka94} or \cite[\S4.2 \& 5.3]{Lu94a}.

Let $w = vv'$ such that $\ell(w) = \ell(v) +\ell(v')$.  We have a unique $\U^+$-module homomorphism
\[
	\pi^w_v: {}^\lambda  V_w \longrightarrow {}^\lambda  V_v, \quad \eta_{w\lambda} \mapsto \eta_{v\lambda}.
\]

\begin{lemma}\cite[Proposition~4.1]{Ka93}\label{lem:reduction}
The map $\pi^w_v$ sends ${\bf B}_w(\lambda)$ to ${\bf B}_v(\lambda) \cup \{0\}$.  In particular, $\pi^w_v (\eta_{w \lambda}) = \eta_{v \lambda}$.
\end{lemma}

\begin{lemma}\label{lem:Splem}
	Let $ b \in \bf{B}(\lambda)$ such that $E_i b =0$. Let ${\rm wt}(b)$ be the weight of $b$ and $n = \langle \alpha_i, {\rm wt}(b)\rangle \ge 0$. Then $F_i^{(n)} b \in B(\lambda).$
\end{lemma}

\begin{proof}
Let $\widetilde{F_i}$ be the Kashiwara's operator defined in \cite{Ka91}. 
Since $E_i b =0$, we have
\[
	F_i^{(n)} b = \widetilde{F_i}^n b.
\]
The lemma then follows from \cite[Theorem~19.3.5]{Lu94}.
\end{proof}

Let $\mathcal{A} = \mathbb{Z}[v,v^{-1}]$. We denote by ${}^\lambda _\mA V$ the $\mA$-form of $\Vla$, which is the free $\mA$-module spanned by $\B(\lambda)$. We similarly define ${}^\lambda_\mA V_w$ to be the free $\mA$-submodule of ${}\Vla_w$ spanned by $\B_w(\lambda)$ for $w \in W$.

Let $\kk$ be a field. We consider the ring homomorphism $\mA \to \kk, v\mapsto 1$. We then define 
\[
	\Vla(\kk) = {}^\lambda _\mA V \otimes_{\mA} \kk, \qquad {}^\lambda V_w(\kk)= {}^\lambda_\mA V_w \otimes_{\mA} \kk.
\]

The set-theoretical map $\pi^w_v: {\bf B}_w(\lambda)\to{\bf B}_v(\lambda) \cup \{0\}$ in Lemma~\ref{lem:reduction} induces a map 
\[
\pi^w_v: {}^\lambda V_w(\kk) \longrightarrow {}^\lambda V_v(\kk).
\]

Note that we have a natural $G^{min}(\kk)$-action on $\Vla(\kk)$ and a natural $U^+(\kk)$ action on  ${}^\lambda V_w(\kk)$ via exponentiation. 

\begin{definition}\label{def:tildes}
Let $\lambda \in X^+$ and $i \in I$. Following \cite[\S0.4]{Ka94}, we define the bijection 
\[
	\tilde{s}_i : \mathbf{B}(\lambda) \rightarrow \mathbf{B}(\lambda). 
\]
This map extends in a unique way to a $\kk$-linear automorphism of ${}^\lambda V(\kk)$, which we still denote by $\tilde{s}_i$. 
\end{definition}

\begin{cor}\label{cor:tildes}
Let $b \in \mathbf{B}(\lambda)$ such that $x_i(t)\cdot b= b \in {}^\lambda V(\kk)$ for any $t \in \kk$, then $\tilde{s}_i (b) = \dot{s}_i \cdot b \in \mathbf{B}(\lambda)$.
\end{cor}

\begin{remark}
The map $\tilde s_i$ will be used in \S\ref{sec:mrmap} in replacement of $\dot{s}_i$. In the construction there, whenever $\tilde s_i$ is applied, it is applied to the subspace fixed by the action of $E_i$. And the restriction of the map $\tilde s_i$ to that subspace coincides with the map $\CT_i$ defined by Lusztig in \cite[\S 2.5 \& 2.6]{Lu-flag}. 
\end{remark}
\begin{proof}
Let ${\rm wt}(b)$ be the weight of $b$ and $n = \langle \alpha_i, {\rm wt}(b)\rangle \ge 0$. It follows from direct computation that that (cf. \cite[\S2.5]{Lu-flag})
 \[
 \dot{s}_i \cdot b = F_i^{(n)} b.
 \]
The corollary follows from Lemma~\ref{lem:Splem} and the definition of $\tilde{s}_i$.
\end{proof}



\subsection{Expressions and subexpressions}
Let $w \in W$. An {\it expression} of $w$ is a sequence 
$$
\mathbf w=(w_{(0)}, w_{(1)}, \cdots, w_{(n)})
$$
in $W$, such that $w_{(0)}=1$, $w_{(n)}=w$ and for $1 \le j \le n$, $w_{(j-1)} \i w_{(j)}$ is either $1$ or a simple reflection. It is called {\it reduced} if $w_{(j-1)}<w_{(j)}$ for all $1 \le j \le n$. In this case, $n$ equals the length $\ell(w)$ of $w$ and the sequence of simple reflections $\{w_{(1)}, w_{(1)} \i w_{(2)}, \ldots, w_{(n-1)} \i w_{(n)}\}$ is called the factors of this reduced expression. This notion of reduced expression is consistent with the usual notion of reduced expression. 

Now we fix a reduced expression $\mathbf w$ of $w$. Let $\{s_{i_1}, \ldots, s_{i_n}\}$ be the sequence of factors for $\mathbf w$. 
Let $v \le w$.  A {\it subexpression} for $v$ in $\mathbf w$ is an expression $\mathbf v=\{v_{(0)}, v_{(1)}, \cdots, v_{(n)}\}$ such that for $1 \le j \le n$, $v_{(j)} \in \{v_{(j-1)}, v_{(j-1)} s_{i_j}\}$. Moreover,
\begin{itemize}
\item
the subexpression $\mathbf v $ is called {\it distinguished} if for $1 \le j \le n$, $v_{(j)} \le v_{(j-1)} s_{i_j}$; 
\item
the subexpression $\mathbf v$ is called {\it positive} if for $1 \le j \le n$, $v_{(j-1)}<v_{(j-1)} s_{i_j}$. 
\end{itemize}
It is proved in \cite[Lemma 3.5]{MR} that for any reduced expression $\mathbf w$ of $w$ and $v \le w$, there exists a unique positive subexpression for $v$ in $\mathbf w$. We denote it by $\mathbf v_{+}$.


 \subsection{Symmetrizable root data}\label{sec:dot}
Let $\dot{G}^{min}(\kk)$ be the Kac-Moody group associated to  a 
 symmetric root datum 
 \[
 \dot{\mathcal{D}} = (S, \dot{A}, \dot{X}, \dot{Y}, (\a_j)_{j \in S}, (\a^\vee_j)_{j \in S})
 \]
  with a given pinning and an automorphism $\sigma$ of the root datum. We assume the automorphism $\sigma : \dot{G}^{min}(\kk) \rightarrow \dot{G}^{min}(\kk)$ such that 
 \begin{enumerate}
 	\item $\sigma$ preserves the pinning; 
 	\item If $j_1 \neq j_2 \in S$ are in the same orbit of $\sigma : S \rightarrow S$, then $j_1, j_2$ do not form an edge of the Coxeter graph;
	\item $j$ and $\sigma(j)$ are in the  connected component of the Coxeter graph, for any $j \in S$.
 \end{enumerate}
 Such $\sigma$ is called {\it admissible}. By \cite[Proposition~14.1.2]{Lu94}, for any simply connected root datum $ {\mathcal{D}}$, there exists a symmetric simply connected $\dot{G}^{min}$ with a pinning and an admissible automorphism $\sigma : \dot{G}^{min}(\kk) \rightarrow \dot{G}^{min}(\kk)$ such that $(\dot{G}^{min}(\kk) )^\sigma \cong  {G}^{min}(\kk)$ which is compatible with the pinnings of $\dot{G}^{min}(\kk) $ and $ {G}^{min}(\kk)$. We also have 
$\CB(\kk) \cong \dot{\CB}(\kk)^\sigma$.	
	We identify $I$ with the set $\overline{S}$ of the $\sigma$-orbits on $S$.
	
	
	We denote by $\dot{W}$ the Weyl group of $\dot{G}^{min}(\kk)$. The automorphism $\s$ on $\dot{G}^{\min}(\kk)$ induces an automorphism on $\dot{W}$, which we still denote by $\s$. We regard $W$ as a subgroup of $\dot{W}$ via the embedding $ s_i \rightarrow \prod_{p \in i} s_{p}$ for $i \in \overline{S}$, which we denote by $i: W \rightarrow \dot{W}$. This gives an isomorphism $i: W \cong \dot{W}^{\s}$. 
	
Let $\mathbf{v}$ and $ \mathbf{w}$ be reduced expressions of $v \le w$ in $W$. Note that we can extend the expressions $\mathbf{v}$ and $ \mathbf{w}$ simultaneously to expressions of $i(v)$ and $i(w)$ by expanding the simple factors from the set of $\sigma$-orbits $\overline{S}$ to the set $S$ consecutively.  The extensions are clearly not unique in general. We pick any such extensions and denote them by $i(\mathbf{v})$ and $ i(\mathbf{w})$, respectively. Then $i(\mathbf{w})$ is a reduced expression of $i(w) \in \dot W$ and $i(\mathbf{v})$ is a subexpression for $i(v) \in \dot{W}$ in $i(\mathbf{w})$. It is clear $i(\mathbf{v})$ is positive if and only ${\mathbf{v}}$ is positive. 
	
We also define the monoids $\dot \fkU(K)$ and $ \dot{\fkG}(K)$ associated with the root datum $\dot{\mathcal{D}}$. The admissible automorphism $\sigma$ induces automorphisms on $ \dot{\fkG}(K)$ and $\dot{\fkU}(K)$. 

By definition, we have a natural homomorphism of monoids 
\begin{equation}\label{eq:iota}
\iota_{\fkG, \dot \fkG}: \fkG(K) \rightarrow  \dot{\fkG}(K)
\end{equation}
 via $i^a \mapsto  \prod_{p \in i \in \overline{S}}p^a $, $(-i)^a \mapsto  \prod_{p \in i \in \overline{S}}(-p)^a $, $i^a \mapsto  \prod_{p \in i \in \overline{S}}\underline{p}^a $ for $i \in I$ and $a \in K$. It is easy to see that the image is contained in $\dot \fkG(K)^\s$. We shall prove in Theorem~\ref{thm:cellular} (2) that $\iota_{\fkG, \dot \fkG}$ induces an isomorphism $\fkG(K) \cong  \dot{\fkG}(K)^\sigma$. 


\vspace{.4cm}

 {\bf From now on, we shall focus on the flag manifolds. Without loss of generality, we will assume that the root data are simply connected. We shall assume that the root data $\dot{\mathcal{D}}$ and $\mathcal{D}$ are given as above, as well as the relevant admissible automorphism $\sigma$, the Kac-Moody groups, etc. All construction relevant to the root datum $\mathcal{D}$ clearly exists for $\dot{\mathcal{D}}$, which we shall not repeat in general.}  

\subsection{The flag manifold $\CB (\kk)$} 

Let $\kk$ be a field. We have the Bruhat decomposition $G^{\min}(\kk)=\sqcup_{w \in W} B^+(\kk) \dot w B^+(\kk)$ (see, e.g. \cite[Theorem 7.67]{Mar}). Let $$\CB(\kk)=G^{\min}(\kk)/B^+(\kk)$$ be the flag manifold. Then we have 
$$
\CB(\kk)=\sqcup_{w \in W} B^+(\kk) \dot w B^+(\kk)/B^+(\kk).
$$

For any $v, w \in W$, we define 
$$
\mathcal R_{v, w}(\kk)=B^+(\kk) \dot w \cdot B^+(\kk) \cap B^-(\kk) \dot v \cdot B^+(\kk) \subset \CB(\kk).
$$

By \cite[Lemma~7.1.22]{Kum}, $\mathcal R_{v, w}(\kk) \neq \emptyset$ if and only if $v \le w$. 


For any $w, v, v' \in W$ with $w=v v'$ and $\ell(w)=\ell(v)+\ell(v')$, we define a map 
\begin{align*}
\pi^w_v: B^+(\kk) \dot w \cdot B^+(\kk) &\to B^+(\kk) \dot v \cdot B^+(\kk), \\
b \dot w \cdot B^+(\kk) &\mapsto b \dot v \cdot B^+(\kk).
\end{align*}

Let $v, w \in W$ with $v \le w$. We fix a reduced expression $\mathbf w=(w_{(0)}, w_{(1)}, \cdots, w_{(n)})$ of $w$. For any subexpression $\vw=\{v_{(0)}, v_{(1)}, \cdots, v_{(n)}\}$ for $v$ in $\mathbf w$, we define the {\it Deodhar component} 
\begin{align*}
\mathcal R_{\vw, \mathbf w}(\kk)
&=\{B \in \mathcal R_{v, w}(\kk) \vert
\pi^w_{w_{(k)}}(B) \in B^-(\kk) \dot v_{(k)} \cdot B^+(\kk) \text{ for all } k\} \\ 
&=\{B \in \mathcal R_{v, w}(\kk) \vert \pi^w_{w_{(k)}}(B) \in \mathcal R_{v_{(k)}, w_{(k)}}(\kk) \text{ for all } k\}.
\end{align*}

It is proved in \cite[Theorem 1.1]{Deo} that $\mathcal R_{\vw, \mathbf w} \neq \emptyset$ if and only if $\vw$ is a distinguished subexpression of $\mathbf w$.

We also have the following  counterparts associated with the root datum $\dot{\mathcal{D}}$ ($v, w \in \dot{W}$):
\[
	\dot{\CB}(\kk), \dot{ \mathcal R}_{v, w}(\kk), \dot{\mathcal R}_{\vw, \mathbf w}(\kk), \text{ etc.} 
\]

\subsection{Marsh-Rietsch parametrization}
Let $\mathbf w$ be a reduced expression of $w$ with factors $(s_{i_1}, \ldots, s_{i_n})$. For any subexpression $\vw=\{v_{(0)}, v_{(1)}, \cdots, v_{(n)}\}$ for $v$ in $\mathbf w$, we set 
\begin{gather*}
J^+_{\vw }=\{1 \le k \le n \vert v_{(k-1)}<v_{(k)}\}, \\
J^o_{\vw}=\{1 \le k \le n \vert v_{(k-1)}=v_{(k)}\}, \\
J^-_{\vw}=\{1 \le k \le n \vert v_{(k-1)}>v_{(k)}\}.
\end{gather*}

Following \cite[Definition 5.1]{MR}, we define a subset $G_{\mathbf v, \mathbf w}(\kk)$ of $G^{\min}(\kk)$ by 
$$
G_{\vw, \mathbf w}(\kk)=
\left\{g_1 \cdots g_n \middle \vert
\begin{aligned} 
\,\,&g_k \in x_{i_k}(\kk) \dot s_{i_k} \i , && \text{ if } k \in J^-_{\mathbf v }; \\ 
&g_k \in y_{i_k}(\kk^\times), && \text{ if } k \in J^o_{\mathbf v }; \\ 
&g_k=\dot s_{i_k},  && \text{ if } k \in J^+_{\mathbf v }.
\end{aligned}
\right\}
$$ 

It is proved in \cite[Proposition 5.2]{MR} that the map $g \mapsto g \cdot B^+(\kk)$ induces an isomorphism 
\[
\begin{tikzcd}
G_{\vw, \mathbf w}(\kk) \arrow[r, "\simeq"] &\mathcal R_{\vw, \mathbf w}(\kk).
\end{tikzcd}
\]

Associated with the symmetric root datum $\dot{\mathcal{D}}$, we define in the same way $\dot G_{\vw, \mathbf w}(\kk)$ and $\dot{\mathcal  R}_{\vw, \mathbf w}(\kk)$ for  $v, w \in \dot{W}$. We have $\dot G_{\vw, \mathbf w}(\kk) \cong \dot{\mathcal  R}_{\vw, \mathbf w}(\kk)$.

\subsection{Chamber Ansatz} In this section we recall the generalized Chamber ansatz introduced in \cite{MR}. The results in loc.cit. are for reductive groups, but can be easily generalized to Kac-Moody groups. We are mostly interested in the symmetric root datum $\dot{\mathcal{D}}$.

For any $\l \in \dot{X}^+$, let ${}^\lambda \dot{P}(\kk)$ be the set of lines in ${}^\lambda \dot{V}(\kk)$.  We define 
\[
{}^\lambda  \dot P_w(\kk) = \{ [\xi] \in {}^\lambda  \dot P(\kk) \vert \xi = \sum_{b \in  \dot \B_w(\lambda)}\xi_b b \in \dot V_w(\kk) \text{ with }\xi_{\eta_{w\lambda} \neq 0}\}.
\]

We define the $ \dot G^{\min}(\kk)$-equivariant map
\begin{equation}\label{eq:flag}
	\pi_\lambda:  \dot \CB(\kk) \longrightarrow {}^\lambda  \dot P(\kk), \qquad
		g \cdot \dot B^+(\kk)  \mapsto  [g \cdot \eta_\lambda].
\end{equation}
If moreover $\l \in  \dot X^{++}$, then $\pi_\lambda$ is injective. 

Note that the image of $ \dot B^+(\kk) \dot w \cdot  \dot B^+(\kk)$ lies in ${}^\lambda  \dot P_w(\kk)$. For any $w, v, v' \in W$ with $w=v v'$ and $\ell(w)=\ell(v)+\ell(v')$, we have the following commutative digram
\begin{equation}\label{eq:reduction}
\xymatrix{
	 \dot B^+(\kk) \dot w \cdot  \dot B^+(\kk) \ar[d]_-{\pi^{w}_v} \ar[r]^-{\pi_\l} &{}^\lambda  \dot P_w(\kk) \ar[d]^-{\pi^{w}_v} \\
	 \dot B^+(\kk) \dot v \cdot  \dot B^+(\kk) \ar[r]^-{\pi_\l} &{}^\lambda  \dot P_v(\kk),
                }
\end{equation}
where $\pi^w_v : {}^\lambda  \dot P_w(\kk) \rightarrow {}^\lambda  \dot P_v(\kk)$ is induced by the set-theoretical map in Lemma~\ref{lem:reduction}.

\begin{definition}
Let $B \in  \dot B^+(\kk)  \dot{w} \cdot  \dot B^+(\kk)$ and $\xi$ be a nonzero vector in $\pi_{\l}(B) \in {}^\lambda  \dot P(\kk) $. Suppose that $\xi=\sum_{b \in  \dot \B(\lambda) } \xi_b b$. Then $\xi_{\eta_{w\lambda}} \neq 0$. For any $w' \in  \dot W$, we define the chamber minor
\[
	\Delta^{w' \lambda}_{w \lambda} (B) = \xi_{\eta_{w' \lambda}} / \xi_{\eta_{w\lambda}} \in \kk.
\]
Note that $\Delta^{w' \lambda}_{w \lambda}$ is independent of the choice of $\xi$ in $\pi_{\l}(B)$. 
\end{definition}

The following result is proved in \cite[Theorem~7.1]{MR} for root datum of finite types. The same argument works in the general case. 

\begin{prop}\label{prop:MR}
Let $B \in  \dot{\mathcal{R}}_{\vw, \mathbf{w}}(\kk)$. We write $B = g \cdot  \dot B^+(\kk) $ for $g = g_1 g_2 \cdots g_n \in  \dot{G}_{\vw, \mathbf w}(\kk)$, where
	\[
	g_k = 
\begin{cases} 
x_{i_k}(m_k) \dot s_{i_k} \i, &\text{ for } m_k \in \kk, \text{ if } k \in J^-_{\vw}, \\ 
y_{i_k}(t_k), &\text{ for } t_k \in \kk-\{0\},  \text{ if } k \in J^o_{\vw}, \\ 
\dot s_{i_k},  & \text{ if } k \in J^+_{\vw}.
\end{cases}
\]
We write $t_k  = -1$ if $k \in J^-_{\vw}$, and $t_k = 1$ if $k \in J^+_{\vw}$. Then we have 
\begin{enumerate}
	\item 
	\[
	 \Delta_{w_{(k)} \lambda }^{v_{(k)} \lambda } (\pi^{w}_{w_{(k)}}(B)) = \prod^k_{l =1} t_l^{-\langle \alpha^\vee_{i_l}, s_{i_{l+1}}s_{i_{l+2}} \cdots s_{i_{k}} \lambda  \rangle }; 
	\]
\item 
	\[
		t_k = \frac{\prod_{I \ni j \neq i_{k}}  \Delta_{w_{(k)} \omega_{j}}^{v_{(k)} \omega_{j}} ( \pi^{w}_{w_{(k)}}(B) ) ^{-a_{j,i_k}}}{ \Delta_{w_{(k)} \omega_{i_k}}^{v_{(k)} \omega_{i_k}} (\pi^{w}_{w_{(k)}}(B) ) \Delta_{w_{(k-1)} \omega_{i_k}}^{v_{(k-1)} \omega_{i_k}} (\pi^{w}_{w_{(k-1)}}(B) )  }. 
	\]
\end{enumerate}
\end{prop}

\begin{rem}
Note that Proposition~\ref{prop:MR} applies when $g \in G^{min}(\kk) \cong \dot{G}^{min}(\kk)^\sigma$. Therefore we can perform Chamber Ansatz for the group $G^{min}(\kk)$ using the representation theory of the group $\dot{G}^{min}(\kk)^\sigma$.
\end{rem}

%




\section{Flag manifolds over a semifield $K\subset \kk$}\label{sec:3}
 Let $K$ be a semifield contained in a field $\kk$ in this section. We write $K^! = K \sqcup \{0 \}$.


\subsection{The flag manifold ${}^\mathbf{W} \CB(K)$} 
We consider both the symmetrizable root datum ${\mathcal{D}}$ and the symmetric root datum $\dot{\mathcal{D}}$ in this subsection.  

Let $v, w \in W$ with $v \le w$. Let $\mathbf w$ be a reduced expression of $w$ and $\mathbf v_+$ be the positive subexpression for $v$ in $\mathbf w$. Following \cite[\S 11]{MR}, we define a subset $G_{\mathbf v_+, \mathbf w}(K)$ of $G^{\min}(\kk)$ by 
$$
G_{\mathbf v_+, \mathbf w}(K)=
\left\{g_1 \cdots g_n \middle \vert
\begin{aligned} 
&g_k \in y_{i_k}(K), && \text{ if } k \in J^o_{\mathbf v_+}; \\ 
&g_k=\dot s_{i_k},  && \text{ if } k \in J^+_{\mathbf v_+}.
\end{aligned}
\right\}
$$ 

Set 
\begin{equation}\label{eq:Rvw}
\CR_{\mathbf v_+, \mathbf w}(K)=G_{\mathbf v_+, \mathbf w}(K) \cdot B^+(\kk) \subset \CB(\kk).
\end{equation}
Note that $|J^o_{\mathbf v_+}|=\ell(w)-\ell(v)$. Note that the natural map $\pi_{\mathbf{w}}: \kk^{\ell(w)-\ell(v)} \longrightarrow \CR_{\mathbf v_+, \mathbf w}(\kk)$ is bijective. Hence the restriction 
\[
	\pi_{\mathbf{w}}: K^{\ell(w)-\ell(v)} \longrightarrow \CR_{\mathbf v_+, \mathbf w}(K)
\]
is also bijective.

We choose a reduced expression $\mathbf{w}$ for each $w \in W$ and denote by $$\mathbf{W} = \{\mathbf{w} \vert w \in W\}$$ the collection of such choices. Following \cite[Theorem 11.3]{MR} and \cite[\S 4.9]{Lu-Spr}, we define the flag manifold ${}^\mathbf{W}\CB(K)$ over $K$ as follows. 

\begin{definition}
Define 
$$
{}^\mathbf{W} \CB(K)=\sqcup_{v \le w} \mathcal R_{\vplus, \mathbf w}(K) \subset \CB(\kk).
$$
\end{definition}

Note that each Marsh-Rietsch component $\mathcal R_{\vplus, \mathbf w}(K)  \cong G_{\vplus, \mathbf w}(K) \cong K^{\ell(w)-\ell(v)}$ is a cell. Thus the subset ${}^\mathbf{W} \CB(K)$ of the flag manifold $\CB(\kk)$, by definition, admits a decomposition into cells. A priori, the set ${}^\mathbf{W} \CB(K)$ depends on $\mathbf{W}$. Moreover, it is not clear from the definition whether the natural action of $G^{\min}(K) \subset G^{\min}(\kk)$ on $\CB(\kk)$ stabilizes ${}^\mathbf{W} \CB(K)$. 

We also have the following following counterparts associated with the symmetric root datum $\dot{\mathcal{D}}$  (and with $v, w \in \dot{W}$): 
\[
{\dot{\mathbf{W}}}, {}^{\dot{\mathbf{W}}} \dot\CB(K), \dot{G}_{\mathbf v_+, \mathbf w}(K),  \dot{\CR}_{\mathbf v_+, \mathbf w}(K), \text{ etc.}
\]

\subsection{The flag manifold ${}^\lambda \dot{\CB}(K)$}\label{lxk} In this subsection, we only consider the symmetric root datum $\dot{\mathcal{D}}$, where we can apply the positivity results of canonical bases. 


We follow \cite[\S 1.4]{Lu-flag}. For any $\l \in \dot{X}^+$, we have the map $\pi_\lambda:  \dot \CB(\kk) \to {}^\lambda  \dot P(\kk)$ defined in \eqref{eq:flag}. Set 
\begin{gather*}
	{}^\lambda  \dot P(K) = \{ [\xi] \in {}^\lambda  \dot P(\kk) \vert \xi = \sum_{b \in  \dot \B(\lambda)} \xi_b b \text{ with }\xi_b \in K^! \text{ for any } b \in  \dot \B(\l)\}, \\
	{}^\lambda  \dot P^\bullet(K) = {}^\lambda  \dot P(K) \cap  \pi_{\l}(\dot \CB(\kk)) \quad \text{ and } \quad {}^\lambda  \dot \CB(K) = \pi_{\l}^{-1} ({}^\lambda  \dot P^\bullet(K)).
\end{gather*}
For $v \le w \in \dot W$, we also define 
\begin{gather*}
		{}^\lambda  \dot P_{v,w}(K) = \{ [\xi] \in {}^\lambda  \dot P(\kk) \vert \xi = \sum_{b \in  \dot \B_{v,w}(\lambda)} \xi_b b \text{ with } \xi_b \in K^! \text{ and both } \xi_{v \lambda}, \xi_{w \lambda } \neq 0\}, \\
	{}^\lambda  \dot P^\bullet_{v,w}(K) = {}^\lambda  \dot P_{v,w}(K) \cap  \pi_{\l}( \dot \CB(\kk)) \quad \text{ and } \quad {}^\lambda  \dot \CB_{v,w}(K) = \pi_{\l}^{-1} ({}^\lambda  \dot P_{v,w}^\bullet(K)).
\end{gather*}

Note that the natural action of the monoid $ \dot G^{min}(K)$ on ${}^\l  \dot P(\kk)$ stabilizes ${}^\l  \dot P(K)$ and thus stabilizes ${}^\l  \dot P^{\bullet}(K)$ and ${}^\l  \dot \CB(K)$. On the other hand, the subset ${}^\l  \dot \CB(K)$ of the flag manifold $ \dot \CB(\kk)$, a priori, depends on the choice of $\l \in  \dot X^+$. Moreover, it is not clear from the definition whether ${}^\l  \dot \CB(K)$ admits a decomposition into cells. 



 We  first discuss several crucial properties of the set ${}^\lambda  \dot \CB(K)$. The set-theoretical operator $\tilde{s}_i$ acts naturally ${}^\lambda  \dot P(K)$.

\begin{lemma}\label{lem:xi}
Let $[\xi] \in {}^\lambda \dot P^\bullet(K)$ be such that $[x_i(t) \cdot \xi] = [\xi]$ for some $i \in I$ and $t \in K$. Then $ [\dot{s}_i  \cdot \xi] = [\tilde{s}_i(\xi)]\in {}^\lambda \dot P^\bullet(K)$. 

\end{lemma}

\begin{proof}Note that $[\dot{s}_i \cdot \xi]  \in  {}^\lambda \dot P^\bullet(\kk)$, while $[\tilde{s}_i(\xi)]\in {}^\lambda \dot P (K)$. Since ${}^\lambda \dot P^\bullet(K)={}^\lambda \dot P^\bullet(\kk) \cap {}^\lambda \dot P (K)$, it remains to prove that $[\dot{s}_i \cdot \xi]=[\tilde{s}_i(\xi)]$. 

We choose a representative $\xi$ of $[\xi]$ with $\xi = \sum_{b \in \dot \B(\lambda)} \xi_b b$ for $\xi_b \in K^!$. Since $x_i(t)$ is unipotent, we must have $x_i(t) \cdot \xi = \xi$. Recall that the root datum is symmetric.  By the positivity property of the canonical bases, we have 
\[
	x_i(t) \cdot b = (\sum_{i=0}^{\infty} t^nE^{(n)}_i )b=\sum_{b' \in \dot{\B}(\lambda) } d'_{b, b'} b', 
\]
for some $d'_{b, b'} \in \mathbb{Z}_{\ge 0}[t] \subset K^! \text{ with } d'_{b, b} =1$.

Therefore
\[
	x_i(t) \cdot \xi = \sum_{b' \in \dot{\B}(\lambda)} (\sum_{b \in \dot\B(\lambda)} \xi_b d'_{b, b'}) b' =  \sum_{b \in \dot\B(\lambda)} \xi_b b.
\]
Hence we have $\xi_{b'} + \sum_{b' \neq b \in \dot\B(\lambda)} \xi_b d'_{b, b'} = \xi_{b'}$ and thus $\sum_{b' \neq b \in \dot\B(\lambda)} \xi_b d'_{b, b'} = 0$. Since $K$ is semifield, $\xi_b d'_{b, b'}=0$ for all $b' \neq b \in \dot\B(\lambda)$. If $\xi_b \neq 0$, then $d'_{b, b'} = 0$ for all $b' \neq b \in \dot\B(\lambda)$ and thus $x_i(t) \cdot b = b$. 
Thanks to Corollary~\ref{cor:tildes},  we have $[\dot{s}_i \cdot \xi]  = [\tilde{s}_i(\xi)]\in {}^\lambda \dot P^\bullet(K)$.
\end{proof}

Since the map $\pi_{\l}: {}^{\l} \dot \CB(K) \to {}^{\l} \dot P(K)$ is injective for $\l \in \dot X^{++}$, we have the following consequence.
\begin{cor}\label{cor:xi}
Let $\lambda \in \dot X^{++}$ and $B \in {}^\lambda \dot \CB(K)$. If $x_i(t) \cdot B = B$ for some $i \in I$ and $t \in K$, then $\dot{s}_i \cdot B \in {}^\lambda \dot \CB(K)$.
\end{cor}

The following two Lemmas will be used in the proof of the main result in this section. 

\begin{lemma}\label{lem:reducX}
Let $\lambda = \mu + \nu $ for $\mu,\nu \in \dot X^+$. If $B \in {}^\lambda \dot \CB(K)$, then we have $B \in {}^{\mu} \dot \CB(K)$ and  $B \in {}^{\nu} \dot \CB(K)$. 
\end{lemma}
 
\begin{proof}

The map $g \cdot \eta_{\lambda} \mapsto g\cdot \eta_{\mu} \otimes g \cdot\eta_{\nu}$ for $g \in \dot G^{\min}(\kk)$ defines a $\dot{G}^{min}(\kk)$-equivariant map 
$$
f: {}^\lambda \dot V(\kk) \longrightarrow  {}^\mu \dot V(\kk) \otimes_{\kk}  {}^\nu \dot V(\kk).
$$ 
Define a $\kk$-linear projection
$$
h: {}^\mu \dot V(\kk) \otimes_k  {}^\nu \dot V(\kk)  \longrightarrow  {}^\mu \dot V(\kk), \quad \sum_{ b \in \dot \B(\nu)}a_b \otimes b \mapsto \sum_{ b \in \dot \B(\nu)} a_b.
$$ 
Set 
\[
\pi^\lambda_\mu = h \circ f: {}^\lambda \dot V(\kk)  \longrightarrow {}^\mu \dot V(\kk).
\]

Let $\xi =  \sum_{b  \in \dot\B(\lambda) } \xi_b b$ with $\xi_b \in K^!$.
 Recall the root datum is symmetric. Thanks to the positivity  property of canonical bases, we obtain
 \[
 	f(\xi) = \sum_{b' \in \dot \B(\mu),b''  \in \dot \B(\nu) } \xi_{b',b'';b}\,  b' \otimes b'' \text{ for some } \xi_{b',b'';b} \in K^!.
 \]

We have $\pi^\lambda_\mu ( \xi ) = \sum_{b' \in \dot \B(\mu)} (\sum_{b''  \in \dot \B(\nu) } \xi_{b',b'';b})\,  b' $. Recall that $K$ is a semifield. If $\xi \neq 0$, then not all $\xi_{b', b'''; b}$ are $0$ and hence $\pi^\lambda_\mu ( \xi ) \neq 0$. 
 
So we obtain, by restriction,
\[
\pi^\lambda_\mu: {}^\lambda \dot V(K) \longrightarrow  {}^\mu \dot V(K) \text{ and } \pi^\lambda_\mu: {}^\lambda \dot P(K) \longrightarrow  {}^\mu \dot P(K).
\]
We conclude by direction computation that for $B \in {}^\lambda \dot \CB(K)$, 
\[
	\pi^\lambda_\mu \circ \pi_\lambda (B) = \pi_\mu(B) \in {}^\mu \dot P(K).
\]
The lemma follows.
\end{proof}

\begin{lem}\label{lem:reducW}
Let $w = v v' \in \dot{W}$ such that $\ell(w) = \ell(v) + \ell(v')$. If  $B \in \dot B^+(\kk)  \dot{w} \cdot \dot B^+(\kk) \cap  {}^\lambda \dot \CB(K)$, then $\pi^w_{v} (B) \in {}^\lambda \dot \CB(K)$. 
 
\end{lem}

\begin{proof}
By definition, $\pi_{\l}(B)=[\xi]$, where $\xi=\sum_{b \in \dot \B_w(\l)} \xi_b b$ with $\xi_b \in K^!$ for any $b \in \B_w(\l)$.  By definition, $\xi_{\eta_{w \lambda}} \neq 0$. By  \eqref{eq:reduction}, $$\pi_{\l} \circ \pi^w_v(B)=\pi^w_v \circ \pi_{\l}(B)=\pi^w_v([\xi])=[\xi'],$$ where $\xi'=\sum_{b \in \dot \B_w(\l)} \xi_b \pi^w_v(b)$. By Lemma~\ref{lem:reduction}, $\xi' \in {}^{\l} \dot V(K)$ and is nonzero. So $\pi^w_{v} (B) \in {}^\lambda \dot \CB(K)$.
\end{proof}






\subsection{The flag manifold $\dot \CB(K)$} 

In this subsection, we define the flag manifold  $\dot \CB(K)$ for the symmetric root datum $\dot\CD$.
\begin{theorem}\label{thm:flag}
Recall that the semifield $K$ is contained in a field. For any $\l \in \dot X^{++}$ and any choice $\dot{\mathbf{W}}$ of reduced expressions of elements in $\dot{W}$,  we have 
$$
{}^{\dot{\mathbf{W}}} \dot \CB(K)= {}^{\l} \dot \CB(K).
$$
\end{theorem}

%
\begin{proof} Let $w \in \dot W$ and $\mathbf{w} = (w_{(1)}, \dots, w_{(n)}) \in \dot{\mathbf{W}}$ be a reduced expression of $w$. Let $v \in W$ with $v \le w$ and $\mathbf v_+$ be the positive subexpression for $v$ in $\mathbf w$. 

We first show that 

(a) $\dot \CR_{\mathbf v_+, \mathbf w}(K) \subset {}^{\l} \dot \CB(K)$.

Let $g \in \dot{G}_{\mathbf v_+, \mathbf w}(\kk)$. By definition, $g = g_1 g_2 \cdots g_n$, where
	\[
	g_k = 
\begin{cases} 
y_{i_k}(t_k), & \text{ for some } t_k \in K \text{ if } k \in J^o_{\mathbf v_{+}}; \\ 
\dot s_{i_k},  & \text{ if } k \in J^+_{\mathbf v_{+}}.
\end{cases}
\]

Set $B_k=g_k g_{k+1} \cdots g_n \cdot \dot B^+(\kk)$ for $1 \le k \le n+1$. We argue by descending induction on $k$ that 

(b) $B_k \in {}^{\l}\dot  \CB(K)$ for all $k$. 

By definition, $B_{n+1}=\dot{B}^+(\kk) \in {}^{\l} \dot \CB(K)$. Suppose that $B_{k+1} \in {}^{\l}\dot \CB(K)$ for some $k \ge 1$. We show that $B_k \in {}^{\l} \dot \CB(K)$. 

If $k \in J^o_{\mathbf v_{+}}$, then $g_k \in y_{i_k}(K)$ and thus $B_k=g_k \cdot B_{k+1} \in  {}^{\l} \dot \CB(K)$. 

If $k \in J^+_{\mathbf v_+}$, then $g_k=\dot s_{i_k}$ and by  \cite[Lemma~11.8]{MR}, $x_{i_k}(t) \cdot B_{k+1} = B_{k+1}$ for all $t \in \kk$. Now following  Corollary~\ref{cor:xi}, we have $B_k = \dot s_{i_k} \cdot B_{k+1} \in {}^{\l} \dot \CB(K)$.  

Thus (b) is proved. 

In particular, $B=B_1 \in {}^{\l} \dot \CB(K)$ and (a) is proved. 

It remains to show that 

(c) ${}^{\l} \dot \CB(K) \cap \dot \CR_{v, w}(\kk) \subset \dot \CR_{\mathbf v_+, \mathbf w}(K)$.  

Let $B \in {}^{\l}\dot  \CB(K) \cap \dot \CR_{v, w}(\kk)$. Then $B \in \dot \CR_{\mathbf v, \mathbf w}(\kk)$ for some subexpression $\mathbf v$ for $v$ in $\mathbf w$. By Lemma~\ref{lem:reducW}, $\pi^{w}_{w_{(k)}}(B) \in  {}^{\lambda}\dot\CB(K)$ for all $1 \le k \le n$ and $j \in I$. By Lemma~\ref{lem:reducX}, $\pi^{w}_{w_{(k)}}(B) \in  {}^{\omega_j}\dot\CB(K)$ for all $1 \le k \le n$ and $j \in I$. In particular, $\Delta_{w_{(k)} \omega_{j}}^{v_{(k)} \omega_{j}} (\pi^{w}_{w_{(k)}}(B)) \in K$ for any $1 \le k \le n$ and $j \in I$. 

Thus for any $k$, 
	\[
		\frac{\prod_{j \neq i_{k}}  \Delta_{w_{(k)} \omega_{j}}^{v_{(k)} \omega_{j}} (\pi^{w}_{w_{(k)}}(B)) ^{-a_{j,i_{k}}}}{ \Delta_{w_{(k)} \omega_{i_k}}^{v_{(k)} \omega_{i_k}} (\pi^{w}_{w_{(k)}}(B)) \Delta_{w_{(k-1)} \omega_{i_k}}^{v_{(k-1)} \omega_{i_k}} (\pi^{w}_{w_{(k-1)}}(B))  } \in K.
	\]
	
By Proposition~\ref{prop:MR} (2), $k \notin J^-_{\mathbf v}$ for all $k$. Hence $\mathbf v=\mathbf v_+$ is the positive subexpression for $v$ in $\mathbf w$. By Proposition~\ref{prop:MR} (2) again, $B=g \cdot B^+$ for some $g \in \dot G_{\mathbf v_+, \mathbf w}(K)$. Hence $B \in \dot \CR_{\mathbf v_+, \mathbf w}(K)$.  

This finishes the proof of the theorem. 
\end{proof}

We shall then simply write $\dot \CB(K) ={}^{\dot{\mathbf{W}}} \dot \CB(K)= {}^{\l} \dot \CB(K)$, which is independent of the choice of $\lambda \in \dot X^{++}$ and the choice of $\dot{\mathbf{W}}$. 


Moreover, for any $v \le w$, we simply write $\dot \CR_{v, w}(K)$ for $\dot{\mathcal{R}}_{v,w}(\kk) \cap \dot \CB(K)$. Then $\dot \CR_{v, w}(K) =\dot{ \mathcal{R}}_{\vplus, \mathbf{w}}(K) = {}^{\lambda}\dot \CB_{v,w} (K) $ which is independent of $\lambda \in \dot X^{++}$ for any reduced expression $\mathbf w$ of $w$ and for any $ \l \in \dot X^{++}$. We also have explicit bijections $\dot \CR_{v, w}(K) \cong K^{\ell(w) - \ell(v)}$ for any reduced expression of $w$.

\subsection{The flag manifold $\CB(K)$}\label{sym-data}
In this subsection, we define the flag manifold  $\CB(K)$ for the symmetrizable root datum $\CD$. The following lemma follows from construction; cf. \cite[Theorem~19.3.5]{Lu94}.
 
	 \begin{lem}\label{lem:sigma}
	 Let $\lambda \in \dot{X}^+$. 
	 There is a unique map $\sigma : {}^\lambda \dot{V}(\kk) \rightarrow {}^{\sigma(\lambda)} \dot{V}(\kk)$ that  maps $\dot{\mathbf{B}}(\lambda)$ to $\dot{\mathbf{B}}(\sigma(\lambda))$ such that $\sigma(g\cdot v) = \sigma(g) \cdot \sigma(v)$.
	 \end{lem}


By the construction, we have the commutative diagram
\[
\xymatrix{ \dot \CB(\kk) \ar[r]^-{\s} \ar[d]_-{\pi_\l} & \dot \CB(\kk) \ar[d]^-{\pi_{\s(\l)}} \\
{}^{\l} \dot P(\kk) \ar[r]^-{\s} & {}^{\s(\l)}\dot P(\kk). 
}
\]

For $\l \in \dot X^{++}$, the map $\s$ maps ${}^{\l} \dot P (K)$ to ${}^{\sigma(\l)} \dot P(K)$. There $\s$ preserves $\dot \CB(K)={}^{\l} \dot \CB(K) = {}^{\sigma(\l)} \dot \CB(K)$. 
We have 
\[
\dot{\CB}(K)^\sigma =  \dot{\CB}(\kk)^\sigma \cap \dot{\CB}(K).
\]
 The subset $\dot{\CB}(K)^\sigma$ admits a natural action of $ {G}^{min}(K) \cong (\dot{G}^{min}(K))^\sigma$. 





\begin{theorem}\label{thm:symmetrizable}
We choose a reduces expression $\mathbf{w}$ for each $w \in W$ and denote by $\mathbf{W} = \{\mathbf{w} \vert w \in W\}$. We have a natural ${G}^{min}(K)$-equivariant bijection
\[
	{}^\mathbf{W} \CB(K) \cong \dot{\CB}(K)^\sigma,
\]
via the embedding $\CB(\kk) \cong \dot{\CB}(\kk)^\sigma \rightarrow  \dot{\CB}(\kk)$.
\end{theorem}

\begin{proof}
We shall identify $\CB(\kk)$ with $\dot{\CB}(\kk)^\sigma$ in this proof to simplify notations. 

Let $v \le w$ in $W$.  Let $\mathbf{w}$ be the reduced expression of $w$ in $\mathbf{W}$ and $\mathbf{v}$ be a subexpression for $v$ in $\mathbf{w}$. Recall the embedding $i: W \rightarrow \dot{W}$ in \S\ref{sec:dot}. We know $i(\mathbf{v})$ is positive if and only ${\mathbf{v}}$ is positive. Therefore we have 
\[
	\mathcal{R}_{v, w}(\kk) \cap \dot{\mathcal{R}}_{i(\mathbf{v}_+), i(\mathbf{w})}(K) = \mathcal{R}_{\vplus, \mathbf{w}}(K).
\]
Hence by Theorem~\ref{thm:flag} we have 
\begin{align*}
	{}^\mathbf{W} \CB(K) &=\sqcup_{v \le w} \CR_{\vplus, \mathbf{w}}(K)=\sqcup_{v \le w} (\mathcal{R}_{v, w}(\kk) \cap \dot{\mathcal{R}}_{i(\mathbf{v}_+), i(\mathbf{w})}(K)) \\ & =\sqcup_{v \le w} (\CB(\kk) \cap \dot{\mathcal{R}}_{i(\mathbf{v}_+), i(\mathbf{w})}(K))=\CB(\kk) \cap \dot{\CB}(K) =\dot{\CB}(K)^\sigma. \qedhere
\end{align*}
\end{proof}

As a consequence, ${}^\mathbf{W} \CB(K) \subset \CB(\kk)$ is independent of the choice $\mathbf{W}$. We shall simply denote this set by $\CB(K)$. Note that $\CB(K)$ is also independent of the choice of the group $\dot{G}^{min}(\kk)$.

For any $v \le w$ in $W$, we set $\CR_{v, w}(K)=\CB(\kk) \cap \dot{\CR}_{i(v), i(w)}(K)=(\dot{\CR}_{i(v), i(w)}(K))^\s$. By the proof of Theorem \ref{thm:symmetrizable}, $\CR_{v, w}(K)=\CR_{\vplus, \mathbf{w}}(K)$ for any reduced expression $\mathbf{w}$. 

\section{Flag manifolds over an arbitrary semifield} 



In this section, we consider an arbitrary semifield $K$ (not necessarily contained in a field). 
 
\subsection{The set ${}^\lambda \dot V(K)$ and ${}^\lambda \dot P(K)$}\label{sec:PK}
We consider only the symmetric root datum $\dot{\mathcal{D}}$ in this subsection. We use various positivity properties of canonical bases in this case. Note that results in \cite[Chap.~22]{Lu94} are stated for the simply laced root data, while the proofs remain valid for the symmetric root data. See also the discussion in \cite[\S5]{Lu-positive}, \cite[\S1.3]{Lu-flag}.

We follow \cite[\S 1]{Lu-flag}. Let $K$ be a semifield and $K^!=K \sqcup \{o\}$, where $o$ is a symbol. Set $a+o=o+a=a, o \times a=a \times o=o$ for all $a \in K^!$. Then the sum and product on $K$ extends to $K^!$ and $K^!$ becomes a monoid under addition and a monoid under multiplication. In the case where $K$ is contained in a field $\kk$, we may take $o$ to be $0 \in \kk$ and $K^!=K \sqcup \{0\} \subset \kk$. 

Let $\l \in \dot X^+$. We define the set of formal sum
\[
{}^{\l} \dot V(K) = \left\{ \xi=\sum_{b \in \dot \B(\lambda)} \xi_b b\quad \middle \vert  \quad 
\begin{aligned}
&\xi_b \in K^!, \\ 
&\xi_b=o \text{ for all but finitely many } b \in \dot \B(\l).
\end{aligned}
\right\} 
\]
Define $\underline o \in \dot V(K)$ by $\underline o_b=o$ for all $b \in \dot \B(\l)$. The set ${}^{\l}\dot V(K)$ is a monoid under addition and has a scalar multiplication $K^! \times {}^{\l}\dot V(K) \to {}^{\l}\dot V(K)$.  Let ${}^\lambda \dot P(K)$ be the set of $(K,\times)$-orbits on ${}^\lambda \dot V(K)-\{\underline o\}$. In the case where $K \subset \kk$, the definition of ${}^{\l} \dot P(K)$ here coincides with the one in \S\ref{lxk}. 


Let $\End({}^{\l}\dot V(K))$ be the set of maps $\zeta: {}^{\l}\dot V(K) \to {}^{\l}\dot V(K)$ such that $\zeta$ commutes with addition and scalar multiplication. Then $\End({}^{\l}\dot V(K))$ is a monoid under composition of maps. By \cite[Proposition 1.5]{Lu-flag}, there is a natural monoid homomorphism $\pi_{\l, K}: \dot  {\mathfrak G} (K) \to \End({}^{\l} \dot V(K))$. In other words, we have a natural monoid representation $\dot \fkG(K) \times {}^{\l} \dot V(K) \to {}^{\l} \dot V(K)$.  In the case where $K \subset \kk$, $\pi_{\l, K}$ is the restriction of the representation $ \dot G^{min} (\kk) \times {}^{\l} \dot V(\kk) \to {}^{\l} \dot V(\kk)$ following Proposition~\ref{prop:GG}. The set-theoretical map $\tilde{s}_i$ in Definition~\ref{def:tildes} clearly induces a map in  $\End({}^{\l}\dot V(K))$, which we denote again by $\tilde{s}_i$.

\subsection{Reduction maps}We still consider only the symmetric root datum $\dot{\mathcal{D}}$ in this subsection. 

Let ${}^{\l} \dot V_w(K) \subset {}^{\l} \dot V(K)$ be the subset consisting of formal sum $\xi=\sum_{b \in \dot \B(\lambda)} \xi_b b$ with $\xi_b \in K^!$ and $\xi_b=o$ if $b \not\in \dot \B_w(\lambda)$.  Let ${}^{\l} \dot  V_{v,w}(K) \subset {}^{\l} \dot  V_w(K)$ be the subset consisting of formal sum $\xi=\sum_{b \in \dot \B(\lambda)} \xi_b b$ with $\xi_b \in K^!$ and $\xi_b=o$ if $b \not\in \dot \B_{v,w}(\lambda)$ with $v \le w$.

Let $v \le w$. We define 
\[
{}^\lambda \dot  P_w(K) =  \{ [\xi] \in {}^\lambda \dot  P(K) \vert \xi  = \sum_{b \in \dot \B_{w}(\lambda)} \xi_b b \text{ such that   } \xi_{\eta_{w \lambda}}  \neq o \},
\]
\[
{}^\lambda \dot P_{v,w}(K) = \{ [\xi] \in {}^\lambda \dot P(K) \vert \xi  = \sum_{b \in \dot \B_{v,w}(\lambda)} \xi_b b \text{ such that both } \xi_{\eta_{w \lambda}}, \xi_{\eta_{v \lambda}} \neq o \}.
\]

\begin{definition}
Let $w = v v'$ such that $\ell(w) = \ell(v) + \ell(v')$. Then the set theoretical map $\pi^w_v: \dot  \B_w(\lambda) \rightarrow \dot  \B_v(\lambda) \cup \{ 0\}$ induces natural map 
\[
	\pi^w_v: {}^{\l} \dot  V_w(K) \longrightarrow {}^{\l} \dot  V_v(K), \qquad \pi^w_v: {}^{\l} \dot  P_w(K) \longrightarrow {}^{\l} \dot  P_v(K)
\]
by identifying $0$ with $\underline{o}$. 
\end{definition}

Let $\mu, \nu \in \dot X^+$. Following \cite[\S4.1]{Lu-flag}, we define the set of formal sum
\[
{}^{\mu,\nu} \dot V(K) =\left\{ \xi= \!\!\!\!  \!\!\!\! \!\!\sum_{(b_1, b_2) \in \dot \B(\mu) \times  \dot \B(\nu)} \!\! \!\!\!\! \!\!\!\!  \xi_{(b_1, b_2)}  (b_1, b_2)    \middle \vert   
\begin{aligned}
&\xi_{(b_1, b_2)} \in K^! \text{ for all } (b_1, b_2) \text{ and }  \\ 
&\xi_{(b_1, b_2)}=o \text{ for all but finitely many } (b_1, b_2)
\end{aligned}
\right\}
\]
Let $\lambda = \mu + \nu$. We have $\dot \fkG(K)$-equivariant maps \cite[\S4.2]{Lu-flag}
\[
 \Gamma : {}^{\lambda} \dot V(K) \rightarrow {}^{\mu,\nu} \dot V(K), \qquad \overline{\Gamma} : {}^{\lambda} \dot P(K) \rightarrow {}^{\mu,\nu} \dot P(K).
\]
We define the projection 
\begin{align*}
	\pi^{\mu,\nu}_{\mu}: {}^{\mu,\nu} \dot P(K) &\rightarrow {}^{\mu} \dot P(K),\\
		\left [\sum_{(b_1, b_2) \in \B(\mu) \times  \B(\nu)} \xi_{(b_1, b_2)}  (b_1, b_2) \right ] &\mapsto \left [\sum_{b_1 \in \B(\mu)} (\sum_{b_2 \in \B(\nu)} \xi_{(b_1, b_2)})  b_1 \right]. 
\end{align*}
\begin{definition}\label{def:reductionX}
We define 
\begin{align*}
\pi^{\lambda}_{\mu} =\pi^{\mu,\nu}_{\mu} \circ \overline{\Gamma}: {}^{\lambda}\dot P(K) &\rightarrow {}^{\mu}\dot P(K).
\end{align*}
\end{definition}
The following lemma is straightforward.
\begin{lem}\label{lem:reductionX2}
The map $\pi^{\lambda}_{\mu}: {}^{\lambda}\dot P(K)  \rightarrow {}^{\mu}\dot P(K)$ is $\dot \fkG(K)$-equivariant. Morevoer, it maps ${}^{\lambda}\dot P_w(K)$ to $ {}^{\mu}\dot P_w(K)$ for any $ w\in \dot{W}$.
\end{lem}


\begin{rem}
When $K \in \kk$ is contained in a field, the map $ \pi^{\lambda}_{\mu} $ is the same as the reduction map defined in Lemma~\ref{lem:reducX}.
\end{rem}

\subsection{Base change}
We consider both the symmetrizable root datum ${\mathcal{D}}$ and the symmetric root datum $\dot{\mathcal{D}}$ in this subsection.  

Let $r: K_1 \to K_2$ be the homomorphism of semifields. Then
\begin{itemize}

\item The maps $i^a \mapsto i^{r(a)}$ for $i \in I$ and $a \in K_1$ induces a monoid homomorphism from ${\fkU}_r: \fkU(K_1) \to \fkU(K_2)$.  

\item The maps $i^a \mapsto i^{r(a)}$, $(-i)^a \mapsto (-i)^{r(a)}$ and $\ui^a \mapsto \ui^{r(a)}$ for $i \in I$ and $a \in K_1$ induces a monoid homomorphism from ${\fkG}_r: \fkG(K_1) \to \fkG(K_2)$.  
\end{itemize}
For the symmetric root datum $\dot{\mathcal{D}}$, we also have  
\begin{itemize}

\item For any $\l \in \dot X^+$, the map $b \mapsto b$, $a \mapsto r(a)$ for $b \in \dot \B(\l)$ and $a \in K_1$ induces a map $\fkV_r: {}^{\l} \dot V(K_1) \to {}^{\l} \dot V(K_2)$ and a map $\fkP_r: {}^{\l} \dot P(K_1) \to {}^{\l} \dot P(K_2)$. 
\end{itemize}

By \cite[\S 1.6]{Lu-flag}, we have the following commutative diagram for the symmetric root datum $\dot{\mathcal{D}}$: 
\[
\xymatrix{
\dot{\mathfrak G} (K_1) \times {}^{\l} \dot V(K_1)  \ar[rr]^-{\pi_{\l, K_1}} \ar[d]_-{({\fkG}_r, \fkV_r)}  & & {}^{\l} \dot V(K_1) \ar[d]^-{\fkV_r} \\
\dot{\mathfrak G} (K_2) \times {}^{\l} \dot V(K_2)  \ar[rr]^-{\pi_{\l, K_2}}  &  & {}^{\l} \dot V(K_2).
}
\]

Recall that for any semifield $K$, we have $$K=\bigcup_{(K', r) \in \mathfrak I_K} r(K').$$ 

Let $\lambda \in \dot  X^{+}$.  We set 
$$
{}^{\l} \dot  P^{\bullet}(K)=\bigcup_{(K', r) \in \mathfrak I_K} \fkP_r({}^{\l} \dot  P^{\bullet}(K')) \subset {}^{\l} \dot  P(K).
$$

For any $v \le w$ in $\dot W$, we set 
\[
{}^{\l} \dot  P^{\bullet}_{v,w}(K) = {}^{\l} \dot P^{\bullet}(K) \cap  {}^{\l} \dot  P_{v,w}(K).
\] Then 
$$
{}^{\l} \dot  P^{\bullet}_{v, w}(K)=\bigcup_{(K', r) \in \mathfrak I_K} \fkP_r({}^{\l} \dot P^{\bullet}_{v, w}(K')) .
$$

\subsection{Marsh-Rietsch maps}\label{sec:mrmap}
We consider the symmetric root datum $\dot{\mathcal{D}}$ in this subsection.  

Let $v \le w$ in $\dot W$ and $\mathbf{w}$ be a reduced expression of $w \in \dot W$. Let $\lambda \in \dot X^+$. We define a subset $\widetilde{\dot G}_{\vplus, \mathbf{w}} (K)$ of $\End({}^{\l}\dot V(K))$ as follows 
\[
\widetilde{\dot G}_{\mathbf v_+, \mathbf w}(K)=
\left\{g_1 \cdots g_n \middle \vert
\begin{aligned} 
&g_k \in \pi_{\l, K}(y_{i_k}(K)), && \text{ if } k \in J^o_{\mathbf v_+}; \\ 
&g_k=\tilde s_{i_k},  && \text{ if } k \in J^+_{\mathbf v_+}.
\end{aligned}
\right\}
\]

We define the Marsh-Rietsch map 
\begin{align*}
	{}^\lambda \mathfrak{mr}_{\vplus, \mathbf{w}} : K^{\ell(w) - \ell(v)} \rightarrow \widetilde{\dot G}_{\mathbf v_+, \mathbf w}(K) &\longrightarrow {}^{\l} \dot P(K),
\end{align*}
where the first map sends the coordinates in $K^{\ell(w) - \ell(v)}$ to $\pi_{\l, K}(y_{i_k}(K))$ for $ k \in J^o_{\mathbf v_+}$ and the second map sends $g$ to $[g \cdot \eta_{\lambda}]$. Similar maps were studied in \cite[\S5]{GKL} over the semifield $\mathbb{R}_{>0}(t)$.

\begin{lem}\label{lem:4.5}
The image of ${}^\lambda \mathfrak{mr}_{\vplus, \mathbf{w}}$ is ${}^{\l} \dot P_{v,w}^{\bullet}(K) \subset {}^{\l} \dot P^{\bullet}(K)$, which is independent of the choice of the reduced expression $\mathbf{w}$.
\end{lem}
\begin{proof}
Let $(K', r) \in \mathfrak I_K$. Then we have the commutative diagram of base change
\[
	\xymatrix{  \widetilde{\dot G}_{\mathbf v_+, \mathbf w}(K') \ar[d]\ar[r] & {}^{\l} \dot P(K') \ar[d] \\
	 \widetilde{\dot G}_{\mathbf v_+, \mathbf w}(K) \ar[r] & {}^{\l} \dot P(K).
	}
\]

Since $K'$ is contained in a field. By Corollary~\ref{cor:tildes} and the proof of Theorem~\ref{thm:flag}, we have 
\[
{}^\lambda \mathfrak{mr}_{\vplus, \mathbf{w}} ((K')^{\ell(w) - \ell(v)}) = G_{\mathbf v_+, \mathbf w}(K') \cdot [\eta_{\lambda}] = {}^{\l} \dot P_{v,w}^{\bullet}(K').\]
Moreover, the image is independent of $\mathbf{w}$ thanks to Theorem~\ref{thm:flag}. So \begin{align*}
{}^\lambda \mathfrak{mr}_{\vplus, \mathbf{w}} (K^{\ell(w) - \ell(v)})
&=\bigcup_{(K', r) \in \mathfrak I_K} {}^\lambda \mathfrak{mr}_{\vplus, \mathbf{w}} \circ \fkP_r ((K')^{\ell(w) - \ell(v)}) \\ 
&=\bigcup_{(K', r) \in \mathfrak I_K} \fkP_r  \circ {}^\lambda \mathfrak{mr}_{\vplus, \mathbf{w}}((K')^{\ell(w) - \ell(v)})  \\ 
&=\bigcup_{(K', r) \in \mathfrak I_K}  \fkP_r({}^{\l} \dot P_{v,w}^{\bullet}(K')) \\ 
&={}^{\l}  \dot P_{v,w}^{\bullet}(K). \qedhere
\end{align*}
\end{proof}

Let $\lambda \in  \dot X^{++}$ and $B  \in {}^{\lambda} \dot P_w(K)$. Let $\mathbf{w}$ be a reduced expression of $w$. Suppose that the image of $B$ in ${}^{\omega_i} \dot P(K)$ under the composition of maps 
\[
\xymatrix{ {}^{\lambda} \dot P_w(K) \ar[r]^-{\pi^w_{w_{(k)}}} & {}^{\lambda} \dot P_{w(k)}(K) \ar[r] &  {}^{\lambda} \dot P(K) \ar[r]^-{\pi^\lambda_{\omega_i}}&  {}^{\omega_i} \dot P(K)
}
\]
is $[\xi]$, where $\xi= \sum_{b \in \dot\B(\omega_i)} \xi_b b$ with $\xi_b \in K^!$. By Lemma~\ref{lem:reductionX2}, $\xi_{\eta_{w_{(k)} \omega_i}} \in K$. Define the chamber minor by
\begin{equation}\label{eq:MR}
	\Delta_{w_{(k)} \omega_i}^{v_{(k)} \omega_i} (\pi^{w}_{w_{(k)}}(B)) = \xi_{\eta_{v_{(k)}\omega_i}} / \xi_{\eta_{w_{(k)}\omega_i}} \in K^!.
\end{equation}

\begin{lem}\label{lem:mrinj}
If $\lambda \in \dot X^{++}$, then ${}^\lambda \mathfrak{mr}_{\vplus,\mathbf{w}}$ is injective.
\end{lem}

\begin{proof}
Let $(K', r) \in \mathfrak I_K$. Recall that $\omega_p$ denotes the fundamental weight in $\dot{X}$ for $p \in S$. We consider the following map: 
\[
	\xymatrix@C+2pc{
		(K')^{\ell(w) - \ell(v)} \ar[r]^-{{}^\lambda \mathfrak{mr}_{\vplus,\mathbf{w}}} &  {}^{\l}  \dot  P^\bullet(K') \ar[r]^-{(\pi^{\l}_{\o_p})_{p \in S}} & \prod_{p \in S} {}^{\omega_p}  \dot  P^\bullet(K') \ar[r]^-{\mathfrak{ca}}  &(K')^{\ell(w) - \ell(v)}.
	}
\]
Here $\mathfrak{ca}$ is the Chamber Ansatz map defined via Proposition~\ref{prop:MR} (2). By Proposition~\ref{prop:MR}, the composition $\mathfrak{ca} \circ (\pi^{\l}_{\o_p})_{p \in S} \circ {}^\lambda \mathfrak{mr}_{\vplus,\mathbf{w}}$ is the identity map on $(K')^{\ell(w) - \ell(v)}$. 

By the explicit formula, the Chamber Ansatz map can be defined for arbitrary semifield and is compatible with base change. We have the following commutative diagram 
\[
\xymatrix@C+2pc{
		(K')^{\ell(w) - \ell(v)} \ar[r]^-{{}^\lambda \mathfrak{mr}_{\vplus,\mathbf{w}}} \ar[d] & {}^{\l}  \dot  P^\bullet(K') \ar[r]^-{(\pi^{\l}_{\o_p})_{p \in S}} \ar[d]& \prod_{p \in S} {}^{\omega_p}  \dot  P^\bullet(K') \ar[r]^-{\mathfrak{ca}}  \ar[d]&(K')^{\ell(w) - \ell(v)}\ar[d] \\
		K^{\ell(w) - \ell(v)} \ar[r]^-{{}^\lambda \mathfrak{mr}_{\vplus,\mathbf{w}}} &  {}^{\l}  \dot  P^\bullet(K) \ar[r]^-{(\pi^{\l}_{\o_p})_{p \in S}} & \prod_{p \in S} {}^{\omega_p}  \dot  P^\bullet(K) \ar[r]^-{\mathfrak{ca}}  & K^{\ell(w) - \ell(v)}.
	}
\]
In particular, the composition $\mathfrak{ca} \circ (\pi^{\l}_{\o_p})_{p \in S} \circ {}^\lambda \mathfrak{mr}_{\vplus,\mathbf{w}}$ is the identity map when restricting to the image of $(K')^{\ell(w) - \ell(v)}$ under the base change $r$. As $(K', r)$ runs over all the elements in $\mathfrak I_K$, the image of $(K')^{\ell(w) - \ell(v)}$ covers the whole space $K^{\ell(w) - \ell(v)}$. Hence ${}^\lambda \mathfrak{mr}_{\vplus,\mathbf{w}}$ is injective on $K^{\ell(w) - \ell(v)}$. 
\end{proof} 
	
	\begin{cor}\label{cor:4.7}
Let $\lambda \in \dot{X}^{++}$ and $ \dot {\mathbf{W}}$ be a collection of reduced expressions of elements in $ \dot  W$. Then the map 
\[
	{}^{\lambda,  \dot {\mathbf{W}}} \mathfrak{mr} = \sqcup_{v \le w} {}^\lambda \mathfrak{mr}_{\vplus, \mathbf{w}} : \sqcup_{v\le w} K^{\ell(w) - \ell(v)} \longrightarrow {}^\lambda  \dot  P^\bullet (K)
\]
is a bijection. 
\end{cor}	

\subsection{The flag manifold $\CB(K)$} \label{sec:symmetrizable}
We consider the symmetrizable root datum ${\mathcal{D}}$ based established results related with the symmetric root datum $\dot{\mathcal{D}}$.  

Let $\lambda \in \dot{X}^{++}$. The set-theoretical map $\sigma: \dot \B(\lambda) \rightarrow  \dot \B(\sigma(\lambda))$ in Lemma~\ref{lem:sigma} induces $\sigma: {}^{\lambda}  \dot P(K) \rightarrow {}^{\sigma(\lambda)}  \dot P(K)$. It is clear that $\sigma$ commutes with the base change map. So 
\begin{align*}
\s({}^{\lambda}  \dot P^\bullet(K)) 
&=\bigcup_{(K', r) \in \mathfrak I_K} \s \circ \fkP_r ({}^{\lambda}  \dot P^\bullet(K'))=\bigcup_{(K', r) \in \mathfrak I_K} \fkP_r \circ \s({}^{\lambda}  \dot P^\bullet(K')) \\ 
&=\bigcup_{(K', r) \in \mathfrak I_K}\fkP_r ({}^{\s(\lambda)}  \dot P^\bullet(K'))={}^{\sigma(\lambda)}  \dot P^\bullet(K).
\end{align*}

Now we state the main theorem. 
\begin{theorem}\label{thm:flagK}
Let $K$ be an arbitrary semifield. We keep the notation in \S\ref{sec:dot}. 

(1) The sets  ${}^{\lambda}  \dot  P^\bullet(K)^\s$ and ${}^{\lambda}  \dot  P_{i(v),i(w)}^\bullet(K)^\s$ (for $v \le w$ in $W$) are independent of the choice of $\lambda \in  (\dot  X^{++})^\s$, which we shall simply denote by $\CB(K)$ and $\CR_{v,w}(K)$, respectively. 

(2) The flag manifold $\CB$ admits a canonical partition $\CB(K) = \bigsqcup_{v \le w \text{ in } W} \CR_{v,w}(K)$. Moreover, for any $v \le w$ in $W$, the Marsh-Rietsch map ${}^{\l} \mathfrak{mr}_{\vplus, \mathbf{w}}: K^{\ell(w)-\ell(v)} \to {}^{\lambda}  \dot  P_{i(v),i(w)}^\bullet(K)^\s=\CR_{v, w}(K)$ is a bijection for any reduced expression of $w$.

(3) The restriction of the monoid action $\fkG(K) \times {}^{\l} \dot P(K) \to {}^{\l} \dot P(K)$ for $\l \in (\dot X^{++})^\s$ gives a monoid action of $\fkG(K) \times \CB(K) \to \CB(K)$. Moreover, this monoid action is independent of the choice of $\l \in (\dot X^{++})^\s$. 
		
(4) For any semifield homomorphism $r: K_1 \to K_2$, we have the following commutative diagram
\[
\xymatrix{
\fkG(K_1) \times \CB(K_1) \ar[d]_-{(\fkG_r, \CB_r)} \ar[r] & \CB(K_1) \ar[d]^-{\CB_r} \\
\fkG(K_2) \times \CB(K_2) \ar[r] & \CB(K_2). 
}
\]
\end{theorem}
	
	\begin{rem}
		Lusztig in \cite{Lu-par} gave another definition of the flag manifold over $K$, which admits a natural action of $\fkG(K)$. It is easy to see that $\CB(K)$ we defined here can be realized as a subset of the one defined in \cite{Lu-par}. They coincide in the case when the semifield $K$ is contained in a field. However, we do not know if they coincide in general. 
	\end{rem}
	
\begin{proof}
(1) Let $\lambda_1 = \lambda_2+ \nu$ such that both $\lambda_1, \lambda_2 \in  (\dot  X^{++})^\s$ and $\nu \in  (\dot  X^+)^\s$. Then we have the commutative diagram
		\[
		\xymatrix{
		& \bigsqcup_{v\le w} K^{\ell(w) - \ell(v)} \ar[ld]_-{{}^{\lambda_1, \mathbf{W}} \mathfrak{mr}} \ar[rd]^-{{}^{\lambda_2, \mathbf{W}} \mathfrak{mr}} & \\
		{}^{\lambda_1} \dot  P^\bullet(K) \ar[rr]^-{\pi^{\l_1}_{\l_2}} & & {}^{\lambda_2}  \dot  P^\bullet(K)
		} 
		\]
		
Moreover, $\pi^{\l_1}_{\l_2}$ commutes with $\s$. Therefore the sets  ${}^{\lambda}  \dot  P^\bullet(K)^\s$ and ${}^{\lambda}  \dot  P_{v,w}^\bullet(K)^\s$ (for $v \le w$ in $W$) are independent of the choice of $\lambda \in  (\dot  X^{++})^\s$.

(2) For $\l \in (\dot X^{++})^\s$, by Lemma \ref{lem:4.5} and Corollary \ref{cor:4.7}, we have ${}^{\l} \dot P^{\bullet}(K)=\sqcup_{v' \le w' \text{ in } \dot W} {}^{\l} \dot P^{\bullet}_{v', w'}(K)$. It is easy to see that $\s({}^{\l} \dot P^{\bullet}_{v', w'}(K))={}^{\l} \dot P^{\bullet}_{\s(v'), \s(w')}(K)$. So 
$$
{}^{\l} \dot P^{\bullet}(K)^\s=\bigsqcup_{v \le w \text{ in } W} {}^{\l} \dot P^{\bullet}_{i(v), i(w)}(K)^\s.
$$
 As we have proved in part (1), this partition is independent of the choice of $\l \in (\dot X^{++})^\s$. 

Let $v \le w$ in $W$. Let $\mathbf{w}$ be a reduced expression of $w$. By the definition of the Chamber Ansatz map, there exists a unique map $\iota$ such that the following diagram commutes
\[
\xymatrix@C+1pc{
		K^{\ell(w) - \ell(v)} \ar[r]^-{{}^\lambda \mathfrak{mr}_{\vplus,\mathbf{w}}} \ar@{^{(}->}[d] & {}^{\l}  \dot  P_{i(v), i(w)}^\bullet(K)^\s \ar@{-->}[rr]^-{\iota} \ar@{^{(}->}[d]& & K^{\ell(w) - \ell(v)}\ar@{^{(}->}[d] \\
		K^{\ell(i(w)) - \ell(i(v))} \ar[r]^-{{}^\lambda \mathfrak{mr}_{\vplus,\mathbf{w}}} &  {}^{\l}  \dot  P_{i(v), i(w)}^\bullet(K) \ar[r]^-{(\pi^{\l}_{\o_p})_{p \in S}} & \prod_{p \in S} {}^{\omega_p}  \dot  P^\bullet(K) \ar[r]^-{\mathfrak{ca}}  & K^{\ell(i(w)) - \ell(i(v))}.
	}
\]

By the proof of Lemma \ref{lem:mrinj}, the composition $\mathfrak{ca} \circ (\pi^{\l}_{\o_p})_{p \in S} \circ {}^\lambda \mathfrak{mr}_{\vplus,\mathbf{w}}$ is the identity map on  $K^{\ell(i(w)) - \ell(i(v))}$. Moreover, by Lemma \ref{lem:4.5} and Lemma  \ref{lem:mrinj}, the map ${}^\lambda \mathfrak{mr}_{\vplus,\mathbf{w}}: K^{\ell(i(w)) - \ell(i(v))} \to {}^{\l}  \dot  P_{i(v), i(w)}^\bullet(K)$ is surjective. Hence the map $\mathfrak{ca} \circ (\pi^{\l}_{\o_i})_{i \in I}: {}^{\l}  \dot  P_{i(v), i(w)}^\bullet(K) \to K^{\ell(i(w)) - \ell(i(v))}$ is bijective. 

Hence the map $\iota: {}^{\l}  \dot  P_{i(v), i(w)}^\bullet(K)^\s \to K^{\ell(w) - \ell(v)}$ is injective and the composition $\iota \circ {}^\lambda \mathfrak{mr}_{\vplus,\mathbf{w}}$ is the identity map on  $K^{\ell(i(w)) - \ell(i(v))}$ is the identity map on $K^{\ell(w)-\ell(v)}$. Thus the map 
${}^{\l} \mathfrak{mr}_{\vplus, \mathbf{w}}: K^{\ell(w)-\ell(v)} \to {}^{\lambda}  \dot  P_{i(v),i(w)}^\bullet(K)^\s$ is bijective. 

(3) Let $(K', r) \in \mathfrak I_K$. By \S\ref{lxk}, the natural action of $\dot G^{\min}(K')$ on ${}^{\l} \dot P(K')$ stabilizes ${}^{\l} \dot P^{\bullet}(K')$. Since this action commutes with $\s$, we have the induced action of $G^{\min}(K')=\dot G^{\min}(K')^\s$ on ${}^{\l} \dot P^{\bullet}(K')^\s$. We have the following commutative diagram 
\[
\xymatrix{
G^{\min}(K') \times {}^{\l} \dot P^{\bullet}(K')^\s \ar@{^{(}->}[r] \ar[d]_-{(\fkG_r, \fkP_r)} & {}^{\l} \dot P^{\bullet}(K')^\s \ar[d]^-{\fkP_r} \\
\fkG(K) \times {}^{\l} \dot P^{\bullet}(K)^\s \ar@{^{(}->}[r] & {}^{\l} \dot P(K). 
}
\]

Since $\fkG(K) \times {}^{\l} \dot P^{\bullet}(K)^\s=\cup_{(K', r) \in \mathfrak I_K} \fkG_r(G^{\min}(K')) \times \fkP_r({}^{\l} \dot P^{\bullet}(K')^\s)$, we have 
\begin{align*}
\fkG(K) \cdot {}^{\l} \dot P^{\bullet}(K)^\s 
&=\bigcup_{(K', r) \in \mathfrak I_K} \fkG_r(G^{\min}(K')) \cdot \fkP_r({}^{\l} \dot P^{\bullet}(K')^\s) \\ &=\bigcup_{(K', r) \in \mathfrak I_K} \fkP_r(G^{\min}(K') \cdot {}^{\l} \dot P^{\bullet}(K')^\s) \\ &=\bigcup_{(K', r) \in \mathfrak I_K} \fkP_r({}^{\l} \dot P^{\bullet}(K')^\s)\\
&={}^{\l} \dot P^{\bullet}(K)^\s.
\end{align*}

This defines a monoid action of $\fkG(K)$ on $\CB(K)$. For $\lambda_1 = \lambda_2+ \nu$ such that both $\lambda_1, \lambda_2 \in  (\dot  X^{++})^\s$ and $\nu \in  (\dot  X^+)^\s$, we have the following commutative diagram
\[
\xymatrix{
\fkG(K) \times {}^{\l_1} \dot P(K) \ar[r] \ar[d]_-{(id, \pi^{\l_1}_{\l_2})} & {}^{\l_1} \dot P(K) \ar[d]^-{\pi^{\l_1}_{\l_2}} \\
\fkG(K) \times {}^{\l_2} \dot P(K) \ar[r]  & {}^{\l_2} \dot P(K).
}
\]

Thus the induced monoid action of $\fkG(K)$ on $\CB(K)$ is independent of the choice of $\l \in (\dot X^{++})^\s$. 

(4) The compatibility of base change follows from part (3) and the following commutative diagram
\[
\xymatrix{
\fkG(K_1) \times {}^{\l} \dot P(K_1) \ar[r]^-{\pi_\l} \ar[d]_-{(\fkG_r, \fkP_r)} & {}^{\l} \dot P(K_1) \ar[d]^-{\fkP_r} \\
\fkG(K_2) \times {}^{\l} \dot P(K_2) \ar[r]^-{\pi_\l}  & {}^{\l} \dot P(K_2).
}
\]
\end{proof}


%
%
%
%
%
%
%
\subsection{Flag manifolds over $\mathbb{R}_{>0}$}We consider the symmetrizable root datum ${\mathcal{D}}$ in this subsection.  

Set 
\[
\CB_n(\BR)=\cup_{w \in W; \ell(w) \le n} B^+(\BR) \dot w B^+(\BR)/B^+(\BR).
\]
By \cite[\S 7.1]{Kum}, $\CB(\BR)$ is endowed with a (unique) projective ind-manifold structure with filtration $\{\CB_n(\BR)\}_{n \ge 0}$ such that the map $i_\l: \CB(\BR) \to {}^{\l} V(\BR)$ is a closed embedding for any $\l \in X^{++}$. The following definition was introduced by Galashin, Karp and Lam in \cite[Definition~10.1]{GKL}. 
\[
\CB_{top}(\BR_{>0}) =  \overline{G^{\min}(\BR_{>0}) \cdot B^+(\BR)/B^+(\BR)} = \overline{U^-(\BR_{>0}) \cdot B^+(\BR)/B^+(\BR)} \subset \CB(\BR).
\]

Now we show that

\begin{theorem}\label{thm:real}
 We have $\CB_{top}(\BR_{>0})=\CB(\BR_{>0})$.  Moreover, the cellular decomposition $\CB(\BR_{>0}) = \sqcup_{v \le w} \CB_{v,w}(\BR_{>0})$ is a decomposition of $\CB(\BR_{>0})$ into topological cells. 
 
\end{theorem}

\begin{proof}
By definition, $U^-(\BR_{>0}) \cdot B^+(\BR)/B^+(\BR)=\sqcup_{w \in W} \CR_{1, w}(\BR_{>0}) \subset \CB(\BR_{>0})$. 

We keep the notation in \S\ref{sym-data}. 
By definition, for any $\l \in \dot X^+$, ${}^{\l} \dot P(\BR_{>0})$ is closed in ${}^{\l} \dot P(\BR)$ and ${}^{\l} \dot \CB(\BR_{>0})$ is closed in $\dot \CB(\BR)$. By Theorem \ref{thm:symmetrizable}, $\CB(\BR_{>0})={}^{\l} \dot \CB(\BR_{>0})^\s$ for any $\l \in (\dot X^{++})^\s$. Therefore $\CB(\BR_{>0})$ is closed in $\dot \CB(\BR)$ and hence also closed in $\CB(\BR)$. Thus $\CB_{top}(\BR_{>0}) \subset \CB(\BR_{>0})$. 

Now we prove the other direction. 

Let $v, w \in W$ with $v \le w$. Let $z \in \CR_{v, w}(\BR_{>0})$. Then by Proposition \ref{prop:cell1} (which is clearly independent of this subsection), $u \cdot z \in \CR_{1, w}(\BR_{>0}) \subset U^-(\BR_{>0}) \cdot B^+(\BR)/B^+(\BR)$ for any $u \in U^+_{v \i}(\BR_{>0})$. Let $v \i=s_{i_1} \cdots s_{i_l}$ be a reduced expression. For any $r \in \BR_{>0}$, we set $u_r=x_{i_1}(r) \cdots x_{i_l}(r)$. Then $\lim_{r \to 0} u_r=1$. Hence $z=\lim_{r \to 0} u_r \cdot z$ lies in the closure of $U^-(\BR_{>0})$. In other words, $\CR_{v, w}(\BR_{>0}) \subset \CB_{top}(\BR_{>0})$. So $\CB(\BR_{>0}) \subset \CB_{top}(\BR_{>0})$. It follows that $\CB_{top}(\BR_{>0})=\CB(\BR_{>0})$.

We apply Proposition~\ref{lem:admis} (which is clearly independent of this subsection) to prove that $\CB_{v,w}(\BR_{>0})$ is a topological cell. It suffices to note that admissible maps are algebraic. Therefore the March-Reich map ${}^{\lambda}\mathfrak{mr}_{\vplus, \mathbf w} : \mathbb{R}^n_{>0} \rightarrow \CB_{v,w}(\BR_{>0})$ is a homeomorphism for any reduced expression $\mathbf{w}$ of $w$.
\end{proof} 

\begin{remark}
If the Kac-Moody root datum is of finite type, then by \cite[Proposition 4.2]{Lus-1},  $U^-(\BR_{>0})$ is the closure of $U^-_{w_I}(\BR_{>0})$, where $w_I$ is the longest element in $W$. In particular, $\CB(\BR_{>0})$ equals to the closure of $U^-_{w_I}(\BR_{>0}) \cdot B^+(\BR)/B^+(\BR)$ in $\CB(\BR)$. The latter one is the original definition of totally nonnegative flag manifold defined by Lusztig in \cite[\S 8.1]{Lus-1}. 
\end{remark} 

We would like to thank Lauren Williams for pointing out the following result to us. 
 \begin{cor}
 The flag manifold $\CB(\BR_{>0})$ is a CW complex. 
 \end{cor}
 
 \begin{proof}
 The same proof in \cite[\S7.3]{RW} applies in our setting. Note that the closure-finite condition follows from Tits system.
 \end{proof}



\section{Cellularity} 

In this section, we prove the cellular decomposition of the monoid $\fkG(K)$ (hence also for $\dot{\fkG}(K)$), as well as the isomorphism $\fkG(K) \cong \dot{\fkG}(K)^\sigma$. In the case the semifield $K \subset \kk$ is contained in a field, they are natural consequences of Tits system.

\subsection{The monoid $\fkG(\{1\})$} We follow \cite[\S 2.11]{Lu-2}. Let $W^\sharp$ be the monoid with generators the symbols $i$ for $i \in I$ and with relations 

(1) For any $i \in I$, $i i=i$;

(2) For any $i \neq j \in I$ with $m_{i j}$ finite, $i j i \cdots=j i j \cdots$ (both products have $m_{i j}$-factors).

Moreover, for any $w \in W$ and a reduced expression $w=s_{i_1} \cdots s_{i_n}$ of $w$, we set $w^\sharp=i_1 \cdots i_n \in W^\sharp$. It is easy to see that $w^\sharp$ is independent of the choice of the reduced expressions of $w$ and the map $$W \to W^\sharp, \qquad w \mapsto w^\sharp$$ is a bijection of sets. 

Following \cite[\S 1.17]{Lu-Spr}, we define the monoid actions of $W^\sharp$ on the set $W$ by
\begin{gather*}
W^\sharp \times W \to W, (w^\sharp, w') \mapsto w \ast w', \text{ where } s_i \ast w'=\max\{w', s_i w'\}; \\
W^\sharp \times W \to W, (w^\sharp, w') \mapsto w \circ_l w', \text{ where } s_i \circ_l w'=\min\{w', s_i w'\}.
\end{gather*}

Now we have the following: 
\begin{itemize}
	\item 	we can naturally identify $W^\sharp$ with $\mathfrak U(\{1\})$ via $ s_i^{\sharp} \to  i^1$;
	\item 	we can naturally identify  $W^\sharp \times W^\sharp$ with $\fkG(\{1\})$  via $(s_i^{\sharp},1^{\sharp}) \to i^1$ and $(1^{\sharp},s_i^{\sharp}) \to (-i)^1$, where $1$ denotes the identity element.
\end{itemize}

 Any semifield homomorphism $r: K_1 \to K_2$ induces  monoid homomorphisms ${\mathfrak U}_r: \mathfrak U(K_1) \to \mathfrak U(K_2)$ and ${\mathfrak G}_r: \mathfrak G(K_1) \to \mathfrak G(K_2)$. In particular, let $r_1: K \to \{1\}$ be the semifield homomorphism sending any element in the semifield $K$ to $1 \in \{1\}$. Then we have monoid homomorphisms
  \begin{gather*}
  {\fkU}_{r_1}: \fkU(K) \to \fkU(\{1\})=W^\sharp,\\
   {\fkG}_{r_1}: \fkG(K) \to \fkG(\{1\})=W^\sharp \times W^\sharp.
  \end{gather*}

For any $w, w'\in W$, we set 
\[
\fkU_w(K)={\fkU}_{r_1} \i(w^\sharp), \qquad \fkG_{w, -w'}(K)={\fkG}^{-1}_{r_1}(w^\sharp, (w')^\sharp).
\]
Let $\fkU_{-w}(K) = \phi(\fkU_{w}(K))$ for any $w \in W$ for the automorphism $\phi$ defined in Lemma~\ref{lem:sym}. It follows that
\[
 \fkG_{w, -w'}(K)= \fkU_{-w'}(K) \cdot \mathfrak{T}(K) \cdot \fkU_w(K) = \fkU_w(K) \cdot \mathfrak{T}(K) \cdot  \fkU_{-w'}(K).
\]

  Then we have  disjoint unions
$$
\fkU(K)=\sqcup_{w \in W} \fkU_w(K), \qquad \fkG(K) = \sqcup_{w,w'} \fkG_{w, -w'}(K).
$$
Moreover we have  
\begin{align*}
\fkU_{w_1}(K) \times \fkU_{w_2}(K) &\longrightarrow \fkU_{w_1 \ast w_2}(K),\\
	 \fkG_{w_1, -w_1'}(K) \times  \fkG_{w_2, -w_2'}(K) &\longrightarrow  \fkG_{w_1 \ast w_2,\,\, -(w'_1 \ast w'_2)}(K),
\end{align*}
under the monoid multiplications.



%




\subsection{Cellularity of $\fkG(K)$}\label{sec:cell} We consider both the symmetrizable root datum ${\mathcal{D}}$ and the symmetric root datum $\dot{\mathcal{D}}$ in this subsection.  

		
%
		Let $w \in W$ and $\mathbf w$ be a reduced expression of $w$. Let $(s_{i_1}, \ldots, s_{i_n})$ be the factors of $\mathbf w$. We define a map 
\begin{equation}\label{eq:Uw}
e_{\mathbf w}: K^{\ell(w)} \to \fkU(K), \qquad (a_1, \ldots, a_n) \mapsto i_1^{a_1} \cdots i_n^{a_n}.
\end{equation}
Then by \cite[\S 2.9]{Lu-2}, the image of $e_{\mathbf w}$ is independent of the choice of the reduced expression of $w$ and the image equals to $\fkU_w(K)$. 

Let $\mathbf{w}$ and $\mathbf{w'}$ be reduced expressions of $w, w' \in W$ with factors $(s_{i_1}, \ldots, s_{i_n})$ and $(s_{i'_1}, \ldots, s_{i'_m})$, respectively. We fix an order $\underline{I}$ of $I$. Define 
\begin{equation*}
\begin{split}
	e_{\mathbf{w}, \underline{I}, -\mathbf{w'}}: K^{\ell(w) +|I|+\ell(w') } &\longrightarrow \fkG_{w,-w'}(K) = \fkU_w(K) \cdot \mathfrak{T}(K) \cdot  \fkU_{-w'}(K), \\
			(a_1, \dots,a_n, b_1, \dots, b_k, c_1, \dots, c_m) &\mapsto  i_1^{a_1} \cdots i_n^{a_n} \cdot \underline{i''_1}^{b_1} \cdots \underline{i''_k}^{b_k} \cdot  (-i')_1^{c_1} \cdots (-i')_n^{c_m}.
\end{split}
\end{equation*}

	\begin{theorem}\label{thm:cellular}
(1)	The maps $e_{\mathbf w}$ and $e_{\mathbf{w}, \underline{I}, -\mathbf{w'}}$ are bijective. In other words, 
	\[
		\fkU(K) = \sqcup\fkU_w(K) \cong \sqcup K^{\ell(w)}, \qquad \fkG(K) = \sqcup\fkG_{w, -w'}(K) \cong \sqcup K^{\ell(w) +|I| +\ell(w') }
	\] are both union of cells.  
	
(2) 	The map $\iota_{\fkG, \dot \fkG}: \fkG(K) \to \dot \fkG(K)$ gives an isomorphism of monoids
	\[
		\fkG(K) \cong \dot\fkG(K)^\sigma.
	\]
	\end{theorem}
 
	\begin{proof}

	For any $w, w' \in W$, we have the following commutative diagram
	\[
	\xymatrix@C+2pc{
	K^{\ell(w)+|I|+\ell(w')} \ar[r]^-{e_{\mathbf{w}, \underline{I}, -\mathbf{w'}}} \ar@{^{(}->}[d] & \fkG_{w, -w'}(K) \ar[d]^-{\iota_{\fkG, \dot \fkG}} \\
	K^{\ell(i(w))+|S|+\ell(i(w'))} \ar[r]^-{e_{i(\mathbf{w}), \underline{S}, -i(\mathbf{w'})}} & \dot \fkG_{i(w), -i(w')}(K). 
	}
	\]

Let $\lambda \in \dot X^{++}$. 
	The following composition is the same as ${}^\lambda \mathfrak{mr}_{1_+, \mathbf{w}}$ hence injective by Lemma~\ref{lem:mrinj},
	\begin{align*}
		\xymatrix{K^{\ell(w)}  \ar[r] & K^{\ell(w) +|S|+\ell(w') } \ar[r] & \dot \fkG_{w,-w'}(K) \ar[r]^-{[\cdot \eta_\lambda]} & {}^\lambda \dot P(K).}
	\end{align*}
	
	Similarly, the following composition is the same as ${}^\lambda \mathfrak{mr}_{1_+, \mathbf{w'}}$ hence injective by Lemma~\ref{lem:mrinj} (recall $\tau$ in Lemma~\ref{lem:sym}),
	\begin{align*}
		\xymatrix{K^{\ell(w')}  \ar[r] & K^{\ell(w)  +|S|+\ell(w')} \ar[r] &\dot \fkG_{w,-w'}(K) \ar[r]^-{\tau}& \dot \fkG_{w,-w'}(K) \ar[r]^-{[\cdot \eta_\lambda]} & {}^\lambda \dot P(K).}
	\end{align*}
	
	Finally, the following composition is injective by examining the coefficient of $\eta_{\lambda}$
		\begin{align*}
		\xymatrix{K^{|S|}  \ar[r] & K^{\ell(w)  +|S|+\ell(w')} \ar[r] & \dot \fkG_{w,-w'}(K) \ar[r]^-{\cdot \eta_\lambda} & {}^\lambda \dot V(K).}
	\end{align*}
	
Therefore the map $e_{i(\mathbf{w}), \underline{S}, -i(\mathbf{w'})}=\iota_{\fkG, \dot \fkG} \circ e_{\mathbf{w}, \underline{I}, -\mathbf{w'}}$ is injective. By \cite[\S 2.9]{Lu-2}, $e_{\mathbf{w}, \underline{I}, -\mathbf{w'}}$ is surjective. Hence $e_{\mathbf{w}, \underline{I}, -\mathbf{w'}}$ is bijective and $\iota_{\fkG, \dot \fkG}$ is injective. 

Now we prove that $\iota_{\fkG, \dot \fkG}(\fkG(K))=\dot \fkG(K)^\s$. By definition, $\iota_{\fkG, \dot \fkG}(\fkG(K)) \subset \dot \fkG(K)^\s$. Also by definition, $\iota_{\fkG, \dot \fkG}(\fkT(K))=\dot \fkT(K)^\s$. It remains to show
	\[
		\iota_{\fkG, \dot \fkG}(\fkU(K))=\dot\fkU(K)^\sigma.
	\]
For any reduced expression $\mathbf{w}$ of $w \in W$, let $i(\mathbf{w})$ be a reduced expression of $i(w)$ extending that of $w$; cf. Theorem~\ref{thm:symmetrizable}.
	Recall $ \dot\fkU(K) = \sqcup _{w \in \dot{W}} \dot\fkU_w(K)$. Since $\sigma (\dot\fkU_w(K)) =  \dot\fkU_{\sigma(w)}(K)$, we have  $\dot\fkU(K)^\sigma \subset  \sqcup _{w \in W} \dot\fkU_{i(w)}(K)^\sigma$.
	
	Consider the following commutative diagram 
	\[
		\xymatrix{K^{\ell(w)} \ar[r]^-{e_{\mathbf{w}}} \ar[d] &  \fkU_{w}(K) \ar[d]\\
		K^{\ell(i(w))} \ar[r]^-{e_{i(\mathbf{w})}} & \dot\fkU_{i(w)}(K)
		}
	\]
	Thanks to the injectivity of $e_{i(\mathbf{w})}$ and $e_{\mathbf{w}} $ in Theorem~\ref{thm:cellular}, we see that any $g \in  \dot\fkU_{i(w)}(K)^\sigma$ must be in the image of $ \dot\fkU_{w}(K)$ by considering the coordinates.
	
	This finishes the proof.
	\end{proof}

\subsection{Flag manifolds over $\{1\}$}\label{sec:flag1}We consider the symmetrizable root datum ${\mathcal{D}}$ in this subsection.  

We denote the action $\fkG(\{1\}) \times \CB(\{1\}) \to \CB (\{1\})$ by $(a, b) \mapsto a \star b$. By
construction, we have the following commutative diagram
\[
\xymatrix{
\fkG(K) \times \CB (K) \ar[r] \ar[d]_-{({\fkG}_r, \CB_r)} & \CB (K) \ar[d]^-{\CB_r} \\
\fkG(\{1\}) \times \CB (\{1\}) \ar[r] & \CB (\{1\}). 
}
\]

In particular, the monoid action restricts to 
\[
\fkG_{x, -y}(K) \times \CB_{v, w}(K) \longrightarrow \CB_{(x, -y) \star (v, w)}(K).
\] 

We compute the map $\star $ using the case where the semifield $K$ is contained in a field $\kk$. We first consider the full flag case. In this finite type case, the following was first obtained by Lusztig in \cite[\S 1.17]{Lu-Spr}.

\begin{prop}\label{prop:cell1}
We have $(x, -y) \star (v, w)=(x \circ_l v, y \ast w)$.
\end{prop}
\begin{proof}
It suffices to prove the claim for symmetric root datum. We shall assume $\mathcal{D}$ is symmetric.

Let $\lambda \in X^{++}$ and $B \in \CB(K)$ .  We have $\pi_\lambda(B) = [\xi] \in {}^{\lambda} P(K)$ for some $\xi= \sum_{b \in \B(\lambda)}\xi_b b$ with $\xi_{b} \in K^!$. 

By Theorem~\ref{thm:flag}, for any $v \le w$ in $W$, $\CR_{v, w}(K)={}^{\l} P^{\bullet}_{v, w}(K)$. In other words, the following conditions are equivalent:
\begin{itemize}
\item $B \in \mathcal{R}_{v,w}(K)$; 

\item $\xi_{b} \neq 0$ only when $v\lambda \le {\rm wt}(b)  \le w\lambda$.
\end{itemize}

For $a \in K$, it then follows from the direct computation and the above equivalence conditions that
$$
x_i(a) \CR_{v, w}(K) \subset \CR_{s_i \circ_l v, w}(K), \quad y_i(a) \CR_{v, w}(K) \subset \CR_{v, s_i \ast w}(K).
$$
This finishes the proof. 
\end{proof}

\section{Admissible functions} \label{sec:5}
 In this section, we prove admissibility of functions arising from  $\CB(K)$ for any semifield $K$. 


\subsection{Admissible functions}
Following \cite[\S 1.2]{Lu-2}, a map $$K^m \to K^{m'}, (a_1, \ldots, a_m) \mapsto (\phi_1(a_1, \ldots, a_m), \ldots, \phi_{m'}(a_1,\ldots, a_m))$$ is called {\it admissible} if for any $i$, $\phi_i$ is of the form $f_i/f'_i$, where $f_i, f'_i \in \BN[x_1, \cdots, x_m]$. A bijective map $K^m \to K^m$ is called {\it bi-admissible} if it is admissible and its inverse is also admissible. The notion can naturally be extended from $K$ to $K^{!}$.


	
\begin{prop}\label{lem:admis}
Let $v \le w$ and $\mathbf{w}$ and $\mathbf{w}'$ be two reduced expressions of $w$. For any $\l \in (\dot X^{++})^\s$, we define the transition map 
\[
	tran_{v, \mathbf{w}, \mathbf{w}'}={}^{\lambda} \mathfrak{mr}_{\vplus, \mathbf{w}}^{-1} \circ {}^{\lambda} \mathfrak{mr}_{\vplus, \mathbf{w'}} : K^{\ell(w) -\ell(v)} \to {}^{\l} \dot P^{\bullet}_{i(v), i(w)}(K)^\s \to K^{\ell(w) -\ell(v)}.
\]
Then $tran_{v, \mathbf{w}, \mathbf{w}'}$ is independent of $\l \in (\dot X^{++})^\s$ and is bi-admissible.
\end{prop}

\begin{proof}
Let $\l, \l' \in (\dot X^{++})^\s$. We have the following commutative diagram
\[
\xymatrix{
& & & {}^{\l} \dot P^{\bullet}_{i(v), i(w)}(K)^\s \\
K^{\ell(w) -\ell(v)} \ar[rr]^-{^{\lambda+\l'} \mathfrak{mr}_{\vplus, \mathbf{w}}} \ar@/^1.5pc/[rrru]^-{^{\lambda} \mathfrak{mr}_{\vplus, \mathbf{w}}} \ar@/_1.5pc/[rrrd]_-{^{\lambda'} \mathfrak{mr}_{\vplus, \mathbf{w}}} & & {}^{\l+\l'} \dot P^{\bullet}_{i(v), i(w)}(K)^\s \ar[ru]_-{\pi^{\l+\l'}_{\l}} \ar[rd]^-{\pi^{\l+\l'}_{\l'}} & \\
& & & {}^{\l'} \dot P^{\bullet}_{i(v), i(w)}(K)^\s.
}
\]
%
%

It follows that the transition map $tran_{v, \mathbf{w}, \mathbf{w}'}$ is independent of the choice of $\l \in (\dot X^{++})^\s$. 
It remains to show that $tran_{v, \mathbf{w}, \mathbf{w}'}$ is bi-admissible. By symmetry, it suffices to prove the map is admissible. 

Let $\lambda \in (\dot X^{++})^\s$. 
We consider 
\[
	\begin{tikzcd}[column sep=large]
		  K^{\ell(w) - \ell(v)} \arrow[r, "^{\lambda} \mathfrak{mr}_{\vplus, \mathbf{w'}} "]  & {}^\lambda \dot V_{v,w}(K) ( \cong (K^!)^{|\dot \B_{i(v),i(w)}(\lambda)\vert})	  \ar[r, "^{\lambda} \mathfrak{mr}^{-1}_{\vplus, \mathbf{w}}"] & 	K^{\ell(w) - \ell(v)}.
		  \end{tikzcd}
\]
The first map $^{\lambda} \mathfrak{mr}_{\vplus, \mathbf{w'}}$ is clearly admissible, since it involves only $\fkG(K)$-action and set-theoretical permutation $\tilde{s}_i$. 
Note that the inverse $^{\lambda} \mathfrak{mr}_{\vplus, \mathbf{w}}^{-1}$ is only defined on the image of $^{\lambda} \mathfrak{mr}_{\vplus, \mathbf{w'}}$. Whenever $^{\lambda} \mathfrak{mr}_{\vplus, \mathbf{w}}^{-1}$ is defined, it can be constructed using the function \eqref{eq:MR} and the generalized Chamber Ansatz Proposition~\ref{prop:MR}. The admissibility follows from the concrete formula. 
\end{proof}



Similarly, we have the following result. 

\begin{prop} Let $w, w', v_1, w_1 \in W$ with $v_1 \le w_1$. We write $(v_2, w_2) = (w, -w') \star (v_1, w_1)$. We choose reduced expressions for $w, w', w_1, w_2$. Then the following action map is admissible and independent of the choice of $\lambda \in (X^{++})^\sigma$:
	\[
		\begin{tikzcd}
			\fkG_{w,-w'}(K)  \times {}^{\lambda} \dot{P}^\bullet_{i(v_1), i(w_1)}(K)^\sigma \arrow[r]  &  {}^{\lambda} \dot{P}^\bullet_{i(v_2), i(w_2)}(K)^\sigma \arrow[d, "{}^\lambda \mathfrak{mr}_{\mathbf{v}_{2,+}, \mathbf{w}_2}^{-1}"] \\
			K^{\ell(w) + |I|+ \ell(w')} \times K^{\ell(w_1) - \ell(v_1)} \arrow[u, "e_{\mathbf{w}, \underline{I}, -\mathbf{w'}} \times {}^\lambda \mathfrak{mr}_{\mathbf{v}_{1,+}, \mathbf{w}_1}"] & K^{\ell(w_2)- \ell(v_2)}
		\end{tikzcd}
	\]

\end{prop}

\subsection{Root datum of finite type} We assume the root datum is of finite type in this subsection.  Keep the notation in \S\ref{sec:dot}. Recall that $w_I \in W$ is the longest element of $W$. It is easy to see that $i(w_I)$ is the longest element of $\dot W$. 


For $\lambda \in \dot X^+$, we write $\lambda' = -i(w_I) (\lambda) \in \dot X^+$. Note that if $\lambda \in \dot X^{++}$ then $\lambda' \in \dot X^{++}$.
 Lusztig \cite[Proposition~21.1.2]{Lu94}  defined a bijection 
 \[
 \phi: \dot \B(\lambda) \longrightarrow \dot \B(\lambda').
 \]
 Now the set-theoretical map $\phi$ induces   maps
	\[
		\phi:{}^{\lambda} \dot V(K) \longrightarrow {}^{\lambda'} \dot V(K), \qquad \phi: {}^\lambda \dot P(K) \longrightarrow {}^{\lambda'} \dot P(K).
	\]
for any semifield $K$. Recall the automorphism $\phi: \dot  \fkG(K) \rightarrow \dot \fkG(K)$ in Lemma~\ref{lem:sym}. By the construction, $\phi (g \cdot z ) = \phi(g) \cdot  \phi  (z)$ for $z\in {}^{\lambda} \dot V(K)$. 

\begin{lem}
Let $\lambda \in \dot  X^{++}$. The map $\phi: {}^\lambda \dot P(K) \longrightarrow {}^{\lambda'} \dot P(K)$ maps  ${}^\lambda \dot P^\bullet(K)$ to ${}^{\lambda'} \dot  P^\bullet(K)$. Moreover, $\phi$ commutes with $\sigma$, that is, we have the following commutative diagram 
\[
\xymatrix{	{}^\lambda \dot P(K) \ar[r]^-{\phi} \ar[d]^-{\sigma} & {}^{\lambda'} \dot P(K) \ar[d]^{\sigma} \\
{}^{\sigma(\lambda)} \dot P(K) \ar[r]^-{\phi}   & {}^{\sigma(\lambda')} \dot P(K)  
}
\]
 
\end{lem} 
\begin{proof}
The first statement is clear when $K$ is contained in a field. Since $\phi$ is set-theoretical, it commutes with base change $r: K' \rightarrow K$. The lemma follows.

To prove the desired commutative diagram, it suffices to prove the following commutative diagram of modules over $\mathbb{Q}(v)$ of the quantum group associated with the root datum $\dot{\mathcal{D}}$
\[
\xymatrix{	{}^\lambda \dot V \ar[r]^-{\phi} \ar[d]^-{\sigma} & {}^{\lambda'} \dot V  \ar[d]^{\sigma} \\
{}^{\sigma(\lambda)} \dot V  \ar[r]^-{\phi}   & {}^{\sigma(\lambda')} \dot V.
}
\]
But this is clearly true, since $\phi$ and $\sigma$ induce commuting automorphism of the relevant quantum group; cf. \cite[Proposition~21.1.2]{Lu94}.
\end{proof}

Take $\l \in \dot X^{++}$ with $\l=-i(w_I)(\l)=\s(\l)$ (such $\l$ always exists, e.g., we may take $\l$ to be the sum of all positive roots in $\dot G$). Then both $\phi$ and $\s$ acts on ${}^{\l} \dot P^{\bullet}(K)$. Since $\phi$ commutes with $\s$, then $\phi$ induces a bijection on $\CB(K) \cong {}^{\l} \dot P^{\bullet}(K)^\s$. Let $v \le w$ in $W$. Then it is easy to see that $\phi$ maps $\CB_{v,w}(K) \cong {}^{\l} \dot P^{\bullet}_{v,w}(K)^\s$ to $\CB_{w w_I, v w_I}(K) \cong {}^{\l} \dot P^{\bullet}_{w w_I, v w_I}(K)^\s$.

%

\smallskip

Finally we prove the admissibility conjectured by Lusztig in \cite[Conjecture~4.4]{Lu-Spr}\footnote{In loc.cit, Lusztig assume that the semifield is $\mathbb{R}_{> 0}$ as the general theory of flag manifold over an arbitrary semifield was not available then.}

\begin{prop}\label{prop:conj}
Let $\l \in \dot X^{++}$ with $\l=-i(w_I)(\l)=\s(\l)$. Let $v \le w$ in  $W$. Set $v'=v w_I$ and $w'=w w_I$. We fix reduced expression $\mathbf w$ for $w$ and $\mathbf v'$ for $v'$. Define the map 
\[
\phi_{\mathbf w, \mathbf v'}:   \xymatrix{
K^{\ell(w) - \ell(v)} \ar[rr]^-{{}^{\l} \mathfrak{mr}_{\mathbf v_+, \mathbf w}} && 
{}^{\l} \dot P^{\bullet}_{v, w}(K)^\s \ar[r]^-{\phi} &
{}^{\l} \dot P^{\bullet}_{w', v'}(K)^\s \ar[rr]^-{{}^{\l}\mathfrak{mr}_{\mathbf w'_+, \mathbf v'} \i} &&
K^{\ell(w) - \ell(v)}}.
\]

Then $\phi_{\mathbf w, \mathbf v'}$ is independent of the choice of $\l$ and is bi-admissible. 
\end{prop}
\begin{remark}
In the case where $K$ is contained in a field, we may replace ${}^{\l} \dot P^{\bullet}_{v, w}(K)^\s$ by $\CB_{v, w}(K)$ and the independence of the choice of $\l$ is automatic. 
\end{remark}
\begin{proof}

Let $\l' \in \dot X^{++}$ with $\l'=-i(w_I)(\l')=\s(\l')$ and $\l' -\l \in \dot X^{++}$. The independence of $\l$ follows from the following commutative diagram
\[
\xymatrix{
& {}^{\l'} \dot P^{\bullet}_{v, w}(K)^\s \ar[dd]_-{\pi^{\l'}_{\l}} \ar[r]^-{\phi} & {}^{\l'} \dot P^{\bullet}_{w', v'}(K)^\s \ar[dd]_-{\pi^{\l'}_{\l}} \ar[dr]^-{{}^{\l'}\mathfrak{mr}_{\mathbf w'_+, \mathbf v'} \i} & \\ 
K^{\ell(w) - \ell(v)} \ar[ru]^-{{}^{\l'} \mathfrak{mr}_{\mathbf v_+, \mathbf w}} \ar[rd]_-{{}^{\l} \mathfrak{mr}_{\mathbf v_+, \mathbf w}} & & & K^{\ell(w) - \ell(v)} \\
& {}^{\l} \dot P^{\bullet}_{v, w}(K)^\s \ar[r]^-{\phi} & {}^{\l'} \dot P^{\bullet}_{w', v'}(K)^\s \ar[ur]_-{{}^{\l}\mathfrak{mr}_{\mathbf w'_+, \mathbf v'} \i} &.
}
\]

It remains to show that $\phi_{\mathbf w, \mathbf v'}$ is bi-admissible. By definition, it is the composition of three bijective maps and hence is also bijective. By symmetry, it suffices to prove that it is admissible.  
Now the lemma follows similar to Proposition~\ref{lem:admis} from the following diagram: 
\[
	\begin{tikzcd}[column sep=large]
		  K^{\ell(w) - \ell(v)} \arrow[r, "{}^\lambda\mathfrak{mr}_{\vplus, \mathbf{w'}} "]  & {}^\lambda V_{v,w}(K) \ar[r, "\phi"] &	 {}^\lambda V_{w',v'}(K)  \ar[r, "{}^\lambda\mathfrak{mr}^{-1}_{\mathbf{w}'_+, \mathbf{v}'}"] & 	K^{\ell(v') - \ell(w')}.
		  \end{tikzcd}
\]
\end{proof}

\end{document}